\newif\ifsvjour

\ifsvjour
\documentclass[a4paper,english,smallextended]{svjour3}

\else
\documentclass[a4paper,11pt,english]{article}
\usepackage{amsthm}
\newtheorem{theorem}{Theorem}
\usepackage{jabbrv,doi}
\fi
\usepackage{geometry}
\usepackage[T1]{fontenc}
\usepackage{babel}
\usepackage{graphicx}
\usepackage{subcaption}
\usepackage{bm}
\usepackage{amsmath,amssymb}
\usepackage{enumerate}
\usepackage[title]{appendix}
\usepackage{chngcntr}
\usepackage{hyperref}

\hypersetup{pdfborder = 0 0 0}
\hypersetup{
    colorlinks=true,
    linkcolor=black,
    citecolor=black,
    filecolor=black,
    urlcolor=black,
}

\usepackage[textwidth=3.5cm,linecolor=orange,bordercolor=orange,backgroundcolor=none]{todonotes}

\newif\ifsubmit
\ifsubmit
\usepackage{tikzexternal,tikzscale}
\graphicspath{{ext}}
\date{}
\else
\date{\today}
\usepackage{tikz,tikzscale,pgfplots}
\usetikzlibrary{external}
\usetikzlibrary{plotmarks}
\pgfplotsset{compat=newest} 
\pgfplotsset{plot coordinates/math parser=false} 
\tikzset{mark size=1/2}
\tikzset{font=\footnotesize}
\fi

\newcommand{\includetikz}[3]{\tikzsetnextfilename{#3}\includegraphics[width=#1,height=#2]{#3}}

\AtEndPreamble{
  \LetLtxMacro{\oldincludegraphics}{\includegraphics}
  \newcommand{\myincludegraphics}[2][]{\tikzsetnextfilename{#2}\oldincludegraphics[#1]{#2}}
  \LetLtxMacro{\includegraphics}{\myincludegraphics}
}

\LetLtxMacro{\oldtodo}{\todo}
\renewcommand{\todo}[2][]{\tikzexternaldisable\oldtodo[#1]{#2}\tikzexternalenable}

\renewcommand{\d}{\mathrm{d}}
\renewcommand{\v}[1]{\bm{#1}}

\renewcommand{\Re}{\operatorname{Re}}
\renewcommand{\Im}{\operatorname{Im}}
\newcommand{\imbrac}[1]{\Im \left[ #1 \right]}

\newcommand{\tzsup}{^{\text{T}}}
\newcommand{\glsup}{^{\text{G}}}
\newcommand{\opint}{\operatorname{I}}
\newcommand{\opquad}{\operatorname{Q}_{n}}
\newcommand{\oprem}{\operatorname{R}_{n}}
\newcommand{\opquadvar}[1]{\operatorname{Q}_{n,#1}}
\newcommand{\opremvar}[1]{\operatorname{R}_{n,#1}}

\newcommand{\contour}{{\mathcal{C}}}
\newcommand{\Res}{\operatorname{Res}}
\newcommand{\resbrac}[1]{\Res\left[ #1 \right]}

\newcommand{\ordo}[1]{{\mathcal{O}\left(#1\right)}}

\newcommand{\KTH}{{Numerical Analysis, Department of Mathematics,\\
KTH Royal Institute of Technology, 100 44 Stockholm, Sweden 
}}

\title{Estimation of quadrature errors in layer potential evaluation
  using quadrature by expansion}

\ifsvjour %
\journalname{Advances in Computational Mathematics}
\author{Ludvig af Klinteberg \and Anna-Karin Tornberg}
\institute{\KTH}

\else %
\usepackage{authblk}

\newcommand{\email}[1]{\href{mailto:#1}{#1}}
\author{Ludvig af Klinteberg\thanks{Email address: \email{ludvigak@kth.se}} }
\author{Anna-Karin Tornberg}
\affil{\KTH}
\fi

\begin{document}

\maketitle

\begin{abstract}
  In boundary integral methods it is often necessary to evaluate layer
  potentials on or close to the boundary, where the underlying
  integral is difficult to evaluate numerically. Quadrature by
  expansion (QBX) is a new method for dealing with such integrals, and
  it is based on forming a local expansion of the layer potential
  close to the boundary.  In doing so, one introduces a new quadrature
  error due to nearly singular integration in the evaluation of
  expansion coefficients.
  Using a method based on contour integration and calculus of
  residues, the quadrature error of nearly singular integrals can be
  accurately estimated.
  This makes it possible to derive accurate estimates for the
  quadrature errors related to QBX, when applied to layer potentials
  in two and three dimensions. As examples we derive estimates for the
  Laplace and Helmholtz single layer potentials.
  These results can be used for parameter selection in practical
  applications.

\ifsvjour
\fi
\end{abstract}

\section{Introduction}
At the core of boundary integral equation (BIE) methods for partial differential equations
(PDEs) lies the representation of a solution $u$ as a layer potential,
\begin{align}
  u(\v x) = \int_\Gamma G(\v x, \v y) \sigma(\v y) \d S_{\v y},
  \label{eq:basic_layer_potential}
\end{align}
where $\Gamma$ is a smooth, closed contour (in $\mathbb R^2$) or
surface (in $\mathbb R^3$), $\sigma(\v y)$ is a smooth density defined
on $\Gamma$, and $G(\v x, \v y)$ is a Green's function associated with
the current PDE of interest. The Green's function is typically
singular along the diagonal $\v x=\v y$, which leads to the integrand
being singular for $\v x\in\Gamma$. We will refer to this as a
singular integral. The field that it produces is however smooth on
each side of $\Gamma$. When $\v x$ is close to $\Gamma$, but not on
it, the layer potential is difficult to evaluate accurately using a
numerical method, even though $G$ is smooth. In this case we say that
the integral is \emph{nearly singular}, since we evaluate $G$ close to
its singularity.

In this paper we discuss the use of residue calculus for estimating the error committed
when computing a nearly singular integral using a quadrature method. The error estimates
are derived in the limit $n\to\infty$, $n$ being the number of discrete quadrature points,
but turn out to be accurate also for moderately large $n$. Throughout we shall use the
symbol $\simeq$ to mean ''asymptotically equal to'', such that
\begin{align}
  a(n) \simeq b(n) \quad  \text{ if } \quad \lim_{n\to\infty} \frac{a(n)}{b(n)} = 1.
\end{align}
The discussion is limited to the Gauss-Legendre rule and the
trapezoidal rule, which are perhaps the two most common quadrature
rules in the BIE field. The kernels that we consider are related to a
new method for singular and nearly singular integration, called
quadrature by expansion (QBX)
\cite{Klockner2013,Barnett2014,Epstein2013}. Specifically, we consider
two classes of kernels that appear when applying QBX to the single
layer potential of Laplace's equation in two and three dimensions,
also referred to as the single layer harmonic potential,
\begin{align}
  G(\v x,\v y) =
  \begin{cases}
    \log |\v x - \v y| & \text{ in $\mathbb R^2$}, \\
    |\v x - \v y|^{-1} & \text{ in $\mathbb R^3$}.
  \end{cases}
  \label{eq:sgl_greens_func}
\end{align}

This paper is organized as follows: In section
\ref{sec:quadr-expans-qbx} we introduce QBX and the relevant kernels
for studying the quadrature errors associated with the method. In
section \ref{sec:quadrature-errors} we discuss the required framework
for estimating quadrature errors, and use it to derive error estimates
for our kernels. Finally, in section \ref{sec:applications} we show
how our results can be used to compute quadrature error estimates that
are useful in practical applications. The results provided in section
\ref{sec:applications} are for model geometries; in appendix
\ref{sec:qbx-spheroid} we include results for a more complex geometry.

\section{Quadrature by expansion (QBX)}
\label{sec:quadr-expans-qbx} 

The central principle of QBX \cite{Klockner2013} is the observation
that a layer potential is smooth away from $\Gamma$, such that it
locally can be represented using a Taylor expansion around an
expansion center $\v x_0$. Given a quadrature rule and a target point
$\v x$ on $\Gamma$, the method can be summarized as:
\begin{enumerate}[1.]
\item Pick an expansion center $\v x_0$ at a distance $r$ from $\v x$ in the normal
  direction, such that $\v x_0$ lies in the quadrature rule's region of high accuracy.
\item Compute a local expansion of the potential around $\v x_0$, truncated to some
  expansion order $p$.
\item Evaluate the local expansion at $\v x$.
\end{enumerate}
The expansion is convergent inside the ball of radius $r$ centered at $\v x_0$, and at the
point of intersection of the ball and $\Gamma$ \cite{Epstein2013}, as illustrated in
Figure \ref{fig:qbx_geo}. We can therefore use QBX to compute the potential both when it
is singular and nearly singular, as long as our target point lies inside the domain of
convergence of a local expansion.
\begin{figure}
  \centering  
  \includegraphics[width=0.5\textwidth]{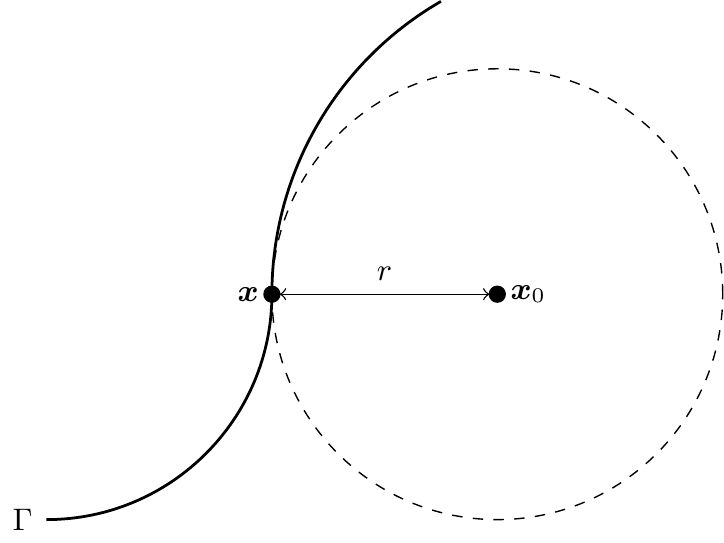}
  \caption{QBX geometry. The local expansion formed at $\v
    x_0$ is valid inside the ball of radius $r$ and at $\v x$.}
  \label{fig:qbx_geo}
\end{figure}

\paragraph{Two-dimensional single layer potential}

In two dimensions it is convenient to let $\mathbb{R}^2=\mathbb{C}$, and work with the
complex version formulation for the single layer potential,
\begin{align}
  u(z) = \Re \int_\Gamma \sigma(w(t))\log(z-w(t)) \d t,
  \label{eq:2d_sgl_layer}
\end{align}
where $w(t)$ is an arc length parametrization of the contour $\Gamma \in
\mathbb{C}$. Forming the local expansion at a center $z_0$ using the series expansion of
$\log(1-\omega)$, $|\omega|<1$, we get \cite{Epstein2013}
\begin{align}
  u(z) = \Re \sum_{j=0}^\infty a_j (z-z_0)^j,
\end{align}
where the expansion coefficients $a_j$ are given by
\begin{align}
  a_0 &=  \int_\Gamma \sigma(w(t)) \log(w(t)) \d t, \\
  a_j &= - \int_\Gamma \frac{\sigma(w(t))}{j(w(t)-z_0)^j} \d t, \quad j>0 .
  \label{eq:def_coeff_2d}
\end{align}
Truncating the expansion to order $p$ and denoting by $\tilde a_j$ the coefficients
computed using a quadrature rule, we define our QBX approximation of the potential as
\begin{align}
  \tilde u(z) = \Re \sum_{j=0}^p \tilde a_j (z-z_0)^j.
\end{align}

\paragraph{Three-dimensional single layer potential}

In three dimensions an expansion of the Green's function about a center $\v x_0$ is formed
using the spherical harmonic addition theorem \cite{Greengard1997},
\begin{align}
  \frac{1}{|\v x - \v y|} = \sum_{l=0}^\infty \frac{4\pi}{2l+1} \sum_{m=-l}^l
  \frac{ |\v x - \v x_0|^l }{ |\v y - \v x_0|^{l+1} }
  Y_l^{-m}(\theta_x,\varphi_x)
  Y_l^{m}(\theta_y,\varphi_y),
\end{align}
where $Y_l^m$ is the spherical harmonic of degree $l$ and order $m$,
\begin{align}
  Y_l^{m}(\theta,\varphi) = \sqrt{\frac{2l+1}{4\pi}\frac{(l-|m|)!}{(l+|m|)!}} 
  P_l^{|m|}(\cos\theta) e^{im\varphi}.
\end{align}
Here $P_l^m$ is the associated Legendre function, and $(\theta_x,\varphi_x)$ and
$(\theta_y,\varphi_y)$ are the spherical coordinates of $\v x-\v x_0$ and $\v y-\v
x_0$. The local expansion is then
\begin{align}
  u(\v x) = \sum_{l=0}^\infty |\v x - \v x_0|^l \sum_{m=-l}^l \alpha_l^m   
  Y_l^{-m}(\theta_x,\varphi_x),
\end{align}
with the expansion coefficients $\alpha_l^m$ given by
\begin{align}
  \alpha_l^m = \frac{4\pi}{2l+1} \int_\Gamma 
  |\v y - \v x_0|^{-l-1}
  Y_l^{m}(\theta_y,\varphi_y)
  \sigma(\v y) \d S_{\v y} .
  \label{eq:def_coeff_3d}
\end{align}
Again truncating the expansion at order $p$ and letting $\tilde\alpha_l^m$ denote
coefficients approximated by quadrature, we get our QBX approximation
\begin{align}
  \tilde u(\v x) = \sum_{l=0}^p |\v x - \v x_0|^l
   \sum_{m=-l}^l \tilde\alpha_l^m 
  Y_l^{-m}(\theta_x,\varphi_x) .
\end{align}

\subsection{Error analysis of QBX}

The error when computing a layer potential using QBX can be divided into two components:
the truncation error and the quadrature error. The truncation error comes from limiting
the local expansion to a finite number of terms, and was analyzed by Epstein et
al. \cite{Epstein2013}. The quadrature error comes from evaluating the integrals
\eqref{eq:def_coeff_2d} and \eqref{eq:def_coeff_3d} using a discrete quadrature rule. To
see the separation of errors we add and subtract the exact expansion coefficients to the
QBX approximation. For the single layer in two dimensions we then get the separation as
\begin{align}
\begin{split}
  u(z) - \tilde u(z) = 
    \underbrace{
      u(z)
      -\Re\sum_{j=0}^p a_j (z-z_0)^j
    }_{\text{truncation error $e_T$}}
    +
    \underbrace{
      \Re\sum_{j=0}^p (a_j - \tilde a_j) (z-z_0)^j
    }_{\text{quadrature error $e_Q$}}    .
\end{split}
\label{eq:qbx_errors_2d}
\end{align}
In \cite[Thm 2.2]{Epstein2013} it is shown that there is a constant $M_{p,\Gamma}$ such
that if $\sigma\in\mathcal{C}^p(\Gamma)$, then
\begin{align}
  |e_T| \le M_{p,\Gamma} \|\sigma\|_{\mathcal{C}^p(\Gamma)} r^{p+1}\log\frac{1}{r} ,
  \label{eq:trunc_err_bound_2d}
\end{align}
where $r$ is the distance from $\Gamma$ to the expansion center. In
three dimensions we similarly have (assuming for the moment $\v
x_0=0$)
\begin{align}
  e_T &= u(\v x) - 
  \sum_{l=0}^p |\v x|^l \sum_{m=-l}^l \alpha_l^m     
  Y_l^{-m}(\theta_x,\varphi_x) ,
  \label{eq:qbx_errors_3d_trunc} \\
  e_Q &=
  \sum_{l=0}^p |\v x|^l \sum_{m=-l}^l 
  (\alpha_l^m - \tilde \alpha_l^m)    
  Y_l^{-m}(\theta_x,\varphi_x)   .
  \label{eq:qbx_errors_3d_quad}
\end{align}
Here a generalization of the results in \cite[Thm 3.1]{Epstein2013} gives that there is a
constant $M_{\Gamma,\delta}$, $\delta>0$, such that if $\sigma$ belongs to the Sobolev
space $H^{3+p+\delta}(\Gamma)$, then
\begin{align}
  |e_T| \le M_{p,\delta} \|\sigma\|_{H^{3+p+\delta}(\Gamma)} r^{p+1}.
  \label{eq:trunc_err_bound_3d}
\end{align}
If we let $h$ be a characteristic length of the discretization of $\Gamma$, and
furthermore keep the ratio $r/h$ constant under grid refinement, then a consequence of the
truncation error estimates is that $e_T$ converges with order $p+1$ under refinement,
\begin{align}
  e_T = \ordo{h^{p+1}} .
\end{align}
In \cite{Epstein2013} and \cite{Klockner2013} it is argued that if the
quadrature error can be maintained at a fixed level $\epsilon$, then
\begin{align}
  |u - \tilde u| = \ordo{\epsilon + h^{p+1}},
\end{align}
such that QBX is convergent with order $p+1$ until hitting the ''floor'' given by
$\epsilon$.

Turning to the quadrature error $e_Q$, one can in practical
applications observe that it grows with the expansion order $p$, such
that there for a given problem exists an optimal $p$ where the total
error $e_T+e_Q$ has a minimum. To control $e_Q$ (and thereby the
minimum error) one can interpolate $\sigma$ to a finer grid before
computing the coefficients, a technique usually referred to as
''upsampling'' \cite{Barnett2014} or ''oversampling''
\cite{Klockner2013}. This works because at large $p$ the difficulties
lie in accurately resolving the kernels in \eqref{eq:def_coeff_2d} and
\eqref{eq:def_coeff_3d}, not $\sigma$ itself. 

The quadrature error $e_Q$ has for the two-dimensional case been
discussed in \cite{Epstein2013} and \cite{Barnett2014}. In
\cite{Epstein2013} an upper bound, which did not explicitly show the
dependence on $p$, was derived for the case when $\Gamma$ is
discretized using Gauss-Legendre panels. In \cite[Thm
3.2]{Barnett2014}, a bound including the $p$-dependence was derived
for $\Gamma$ discretized using the trapezoidal rule.

In what follows of this paper we shall develop quadrature error \emph{estimates}, not
bounds, that accurately predict the order of magnitude of $|a_j-\tilde a_j|$ and
$|\alpha_l^m-\tilde \alpha_l^m|$ in the asymptotic region of $n\to\infty$. We will do this
for the $n$-point trapezoidal and Gauss-Legendre quadrature rules on certain
geometries. To that end, we will consider two different classes of kernels (depicted in
Figure \ref{fig:kernels}),
\begin{align}
  f_p(z) &= \frac{1}{(z-z_0)^p}, &z,z_0 &\in \mathbb{C}, 
  \label{eq:fp_def}
  \\
  g_p(x,y) &= \frac{1}{\left( (x-x_0)^2 + (y-y_0)^2 \right)^p}, &x,y,x_0,y_0 &\in
  \mathbb{R},
  \label{eq:gp_def}
\end{align}
where the respective singularities $z_0$ and $(x_0,y_0)$ of the kernels are assumed to lie
close to, but not on, the interval of integration.  The kernel $f_p$ is relevant when
analyzing the computation of $a_j$, which we will refer to as the complex kernel. The
corresponding kernel for $\alpha_l^m$ is $g_p$, and we will refer to it as the Cartesian
kernel.
\begin{figure}
  \centering
  \includetikz{0.48\textwidth}{0.3\textwidth}{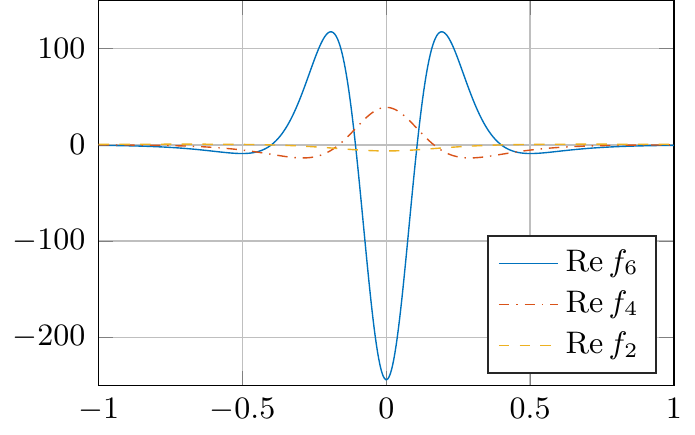}
  \includetikz{0.48\textwidth}{0.3\textwidth}{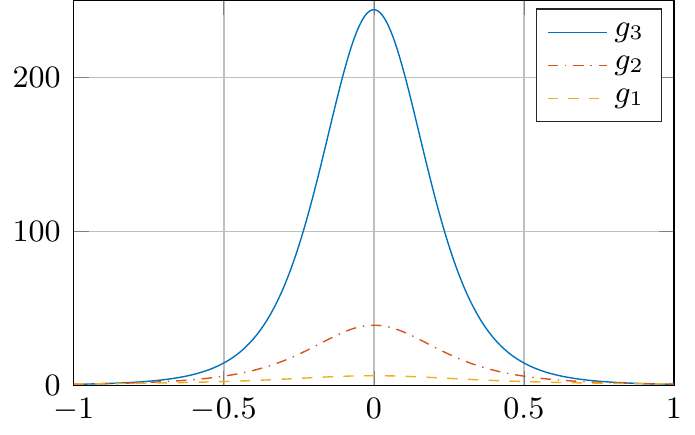}
  \caption{Examples of the kernels $f_p$ \eqref{eq:fp_def} and $g_p$ \eqref{eq:gp_def} on
    $[-1,1]$, with $z_0=i y_0$ and $(x_0,y_0)=(0,0.4)$. The real part of the complex kernel
    $f_p$ is shown to illustrate how it oscillates (in absolute value $|f_{2p}=g_p|$ on
    this interval).}
  \label{fig:kernels}
\end{figure}

\section{Estimating quadrature errors}
\label{sec:quadrature-errors}

The purpose of quadrature is to numerically approximate the definite integral of a
function $f$ over some interval $\Gamma$,
\begin{align}
  \opint[f] = \int_\Gamma f(x) \d x,
\end{align}
where $\Gamma$ is typically a subset of the real line or a closed curve in the complex
plane. This approximation is computed through an $n$-point quadrature rule, using a set of
nodes $x_i \in \Gamma$ and weights $w_i$,
\begin{align}
  \opquad[f] = \sum_{i=1}^n f(x_i) w_i .
\end{align}
For an $n$-point interpolatory quadrature rule, the weights $w_i$ are determined by
integrating the polynomials of degree $n-1$ that interpolate $f$ at the nodes $x_i$. The
error of the approximation is given by the remainder term
\begin{align}
  \oprem = \opint - \opquad,
\end{align}
which converges to zero as $n\to\infty$ unless the function is
ill-behaved in some way, e.g. if $f$ has a singularity on $\Gamma$. As
practitioners of quadrature, we typically want to know the
relationship between $\oprem$ and $f$ in order to make our
computations efficient and reliable. Luckily, most quadrature rules
come equipped with a number of error bounds involving $f$ in one way
or another. However, most such bounds are useful only if $f$ is
analytic in the whole complex plane, or in a large region around
$\Gamma$. In section \ref{sec:class-error-estim} we give an example
where an error bound for the Gauss-Legendre rule wildly overestimates
the actual error.

We now consider what happens when $f$ is meromorphic, i.e. analytic
everywhere except for at a set of poles.  The focus of our present
work is the case when $f$ has poles close to $\Gamma$. Then the best
way of obtaining accurate estimates of $\oprem$ is to use contour
integrals in the complex plane, using the theory of Donaldson and
Elliott \cite{Donaldson1972}. The key principle is that the quadrature
rule $\opquad$ can be connected to a rational function $q_n(z)$, which
has simple poles at the quadrature nodes $x_i$, and residues at those
poles equal to the quadrature weights $w_i$. If $\contour$ is a
contour enclosing $\Gamma$, on and within which the complex
continuation of $f$ is analytic, then
\begin{align}
  \opquad[f] = \frac{1}{2\pi i} \int_\contour f(z)q_n(z) \d z .
\end{align}
There also exists a characteristic function $m(z)$ such that we can express the integral
over $\Gamma$ as a contour integral,
\begin{align}
  \opint[f] = \frac{1}{2\pi i} \int_\contour f(z)m(z) \d z.
\end{align}
We define the remainder function
\begin{align}
  k_n(z) = m(z) - q_n(z),
  \label{eq:def_rem_fcn}
\end{align}
which is analytic in the complex plane with $\Gamma$ deleted. Using $k_n$ we can express
the remainder as a contour integral,
\begin{align}
  \oprem[f] = \frac{1}{2\pi i} \int_\contour f(z)k_n(z) \d z .
  \label{eq:rem_contour}
\end{align}
As $|z|$ tends to infinity, $k_n$ tends to zero at least like
$\mathcal O(|z|^{-n})$ for interpolatory quadrature
\cite{Elliott2008}, so by taking $\contour$ to infinity we see that
interpolatory quadrature is exact for polynomials of degree $<n$. If
$f$ has one or more poles $z_j$ enclosed by $\contour$, we deform the
contour integral to include small circles enclosing those poles (see
e.g. the illustration in \cite{Elliott2008}). Letting the radius of
the circles go to zero, the remainder is given by the integral over
$\contour$ minus the residues at the poles,
\begin{align}
  \oprem[f] = \frac{1}{2\pi i} \int_\contour f(z)k_n(z) \d z -
  \sum_{j} \resbrac{ f(z)k_n(z), z_j }.
  \label{eq:rem_contour_res}
\end{align}
For reference we here state the definition of the residue of a
function $g(z)$ which has an order $p$ pole at $z_0$:
\begin{align}
  \resbrac{ g(z), z_0 } = \frac{1}{(p-1)!} \lim_{z \to z_0}
  \frac{\d^{p-1}}{\d z^{p-1}} \left( (z-z_0)^p g(z) \right).
  \label{eq:def_residue}
\end{align}
If the poles of $f$ are close to $\Gamma$, then the remainder
$\oprem[f]$ is typically dominated by the corresponding residues, and
the contribution from the contour integral is negligible. If
$f(z)k_n(z)$ goes to zero as $\contour$ is taken to infinity for some
$n \ge N$, then the remainder is equal to the sum of the residues for
those $n$.

Given a meromorphic integrand $f$ and a quadrature rule $\opquad$, our ability to estimate
$\oprem$ depends on our knowledge of $k_n(z)$ and our ability to compute the residues of
$f(z)k_n(z)$ at the poles. For some quadrature rules we have closed form expressions for
$k_n(z)$, while for others we only have asymptotic estimates valid for large $n$. In what
follows we shall summarize the relevant formulas for the Gauss-Legendre and trapezoidal
quadrature rules and apply them to our model kernels.

\subsection{The Gauss-Legendre quadrature rule}
The $n$-point Gauss-Legendre quadrature rule \cite[ch. 25]{Abramowitz1972} belongs to the
wider class of Gaussian quadrature, and is extensively used in applications where the
integrand is not periodic. It is by convention defined for the interval $[-1,1]$,
\begin{align}
  \opint[f] = \int_{-1}^1 f(x) \d x .
\end{align}
The weights and nodes of the quadrature rule $\opquad$ are associated with $P_n(x)$, the
Legendre polynomial of degree $n$. The nodes are the roots of $P_n$,
\begin{align}
  P_n(x_i) = 0, \quad i=1,\dots,n,
\end{align}
and ordered, such that $x_i < x_{i+1}$. The weights are given by
\begin{align}
  w_i = \frac{2}{(1-x_i^2)P_n'(x_i)^2} .
\end{align}

In our analysis of the kernels $f_p$ and $g_p$ we will stay on the standard interval
$[-1,1]$ on the real axis. Setting
\begin{align}
  z_0 &= a + ib, &
  (x_0, y_0) &= (a,b),
\end{align}
with
\begin{align}
  -\infty < a < \infty, \quad b > 0,
\end{align}
our target kernels are then
\begin{align}
  f_p(z) &= \frac{1}{(z-z_0)^p}, 
  \label{eq:fp_gl}
  \\
  g_p(x) &= \frac{1}{((x-a)^2 + b^2)^p} .
  \label{eq:gp_gl}
\end{align}

\subsubsection{Classic error estimate}
\label{sec:class-error-estim}
There exists a classic error estimate for the Gauss-Legendre rule, available in e.g.
Abramowitz and Stegun \cite[eq. 25.4.30]{Abramowitz1972}, stating that on an interval of
length $L$ the error is given by
\begin{align}
  |\oprem[f]| \le \frac{L^{2n+1} (n!)^4}{(2n+1)[(2n)!]^3} \|f^{(2n)}\|_\infty .
  \label{eq:gl_classic}
\end{align}
The most important interpretation of this result is that $n$-point Gauss-Legendre
quadrature will integrate polynomials of degree up to $2n-1$ exactly. In practice the
estimate is only useful for very smooth integrands, due to the high derivative of $f$ in
the estimate. Consider the example of $f=g_1$ with $a=0$, which can be found in Brass and
Petras \cite{Brass2011},
\begin{align}
  f(x) = \frac{1}{x^2+b^2}, \quad b > 0, \quad \Gamma = [-1,1].
\end{align}
The norm of the derivative is given by
\begin{align}
  \| f^{(2n)} \|_\infty = |f^{(2n)}(0)| = \frac{(2n)!}{b^{2n+2}} .
\end{align}
Inserting this into the estimate and applying Stirling's formula,
\begin{align}
  \sqrt{2\pi}n^{n+\frac 12}e^{-n} < n! <   2\sqrt{\pi}n^{n+\frac 12}e^{-n},
  \label{eq:stirling}
\end{align} 
we get that
\begin{align}
  |\oprem[f]| \le \frac{4\pi}{b^2(2b)^{2n}} .
\end{align}
This bound goes to infinity exponentially fast for $b < 1/2$, while in
practice Gauss-Legendre quadrature exhibits exponential convergence
for this integrand.

There exists a large number of error estimates involving lower order derivatives of the
integrand, available in the work by Brass et
al. \cite{Brass1996,Brass2011}. Alternatively, one can estimate the error through a
contour integral, which is what we will do in the following section.

\subsubsection{Contour integral}
An expression for estimating the error of Gaussian quadrature as a contour integral was
found by Barrett \cite{Barrett1961} for certain cases, and later generalized by Donaldson
and Elliott \cite{Donaldson1972}. The below derivation follows that of Barrett
\cite{Barrett1961}.

It can be shown for Gauss-Legendre quadrature that the weights are given by
\begin{align}
  w_i = \frac{1}{P_n'(x_i)} \int_{-1}^1 \frac{P_n(t)}{t-x_i} \d t,
\end{align}
where $x_i$ are the zeros of $P_n(x)$. The weights can also be
computed as the residues of the function
\begin{align}
  q_n(z) = \frac{1}{P_n(z)} \int_{-1}^1 \frac{P_n(z) - P_n(t)}{z-t} \d t,
\end{align}
at the nodes $x_i$, 
\begin{align}
  w_i = \Res[q_n, x_i] = \lim_{z\to x_i}(z-x_i)q_n(z).
\end{align}
It follows that
\begin{align}
  \opquad[f] = \sum_{i=1}^n f(x_i) w_i = \frac{1}{2\pi i} \int_\contour f(z)q_n(z) \d z,
\end{align}
where $\contour$ now contains $[-1,1]$. It can also be seen that since
\begin{align}
  f(t) = \frac{1}{2\pi i} \int_\contour \frac{f(z)}{z-t} \d z,
\end{align}
we can write
\begin{align}
  \opint[f] = \frac{1}{2\pi i} \int_\contour f(z)m(z) \d z,
\end{align}
where
\begin{align}
  m(z) = \int_{-1}^1 \frac{\d t}{z-t}
\end{align}
It follows that the remainder function $k_n(z)$ in
\eqref{eq:def_rem_fcn}--\eqref{eq:rem_contour_res} is given by
\begin{align}
  k_n(z) = m(z) - q_n(z) = \frac{1}{P_n(z)} \int_{-1}^1 \frac{P_n(t)}{z-t} \d t .
\end{align}
While we do not have a closed form expression for $k_n(z)$, it can in the limit
$n\to\infty$ be shown to satisfy \cite{Barrett1961,Donaldson1972}
\begin{align}
  k_n(z) \simeq \frac{c_n}{(z+\sqrt{z^2-1})^{2n+1}},
  \label{eq:kn_gl}
\end{align}
where 
\begin{align}
  c_n = \frac{2\pi(\Gamma(n+1))^2}{\Gamma(n+1/2)\Gamma(n+3/2)} \simeq 2\pi.
\end{align}
While this is an asymptotic result valid for $n\to\infty$, we shall see that it provides
an accurate approximation of $k_n$ also for moderately large $n$.

In the following analysis we will need derivatives of $k_n$, for the residue at high order
poles. The first two derivatives are
\begin{align}
  k_n'(z) &\simeq k_n(z)\frac{-(2n+1)}{\sqrt{z^2-1}}, \label{eq:kn1}\\
  k_n''(z) &\simeq k_n(z)\left( \frac{(2n+1)^2}{\sqrt{z^2-1}^2} +
    \frac{z(2n+1)}{\sqrt{z^2-1}^3}\right).
  \label{eq:kn2}
\end{align}
From \eqref{eq:kn1} and \eqref{eq:kn2} we can induce that the $q$th derivative of $k_n$
will have the form
\begin{align}
  k_n^{(q)}(z) \simeq k_n(z)\left( \left( -\frac{2n+1}{\sqrt{z^2-1}} \right)^q +
    \mathcal{O}\left( (2n+1)^{q-1} \right) \right),
\end{align}
such that we for large $n$ can use the approximation
\begin{align}
  k_n^{(q)}(z) \simeq k_n(z) \left( -\frac{2n+1}{\sqrt{z^2-1}} \right)^q .
  \label{eq:knq}
\end{align}

\subsubsection{Complex kernel}
\label{sec:gl-2d-kernel}
We now wish to study the Gauss-Legendre rule applied to the complex kernel
\eqref{eq:fp_gl} on $[-1,1]$. The integral is then
\begin{align}
  \opint[f_p] = \int_{-1}^1 \frac{\d x}{(x-z_0)^p},
\end{align}
which has a pole of order $p$ at $z_0$. Letting $\contour$ in
\eqref{eq:rem_contour_res} enclose $[-1,1]$ and $z_0$, the integrand
$f_p(z)k_n(z)$ vanishes as we take $\contour$ to infinity. The
remainder is given by the residue,
\begin{align}
  \oprem[f_p] = -\resbrac{\frac{k_n(z)}{(z-z_0)^p}, z_0} =
  \frac{-k_n^{(p-1)}(z_0)}{(p-1)!}.
  \label{eq:rem_gl_complex_res}
\end{align}
Using the estimate \eqref{eq:knq} for the derivatives of $k_n$, we get the following estimate for the remainder:
\begin{theorem}
  Let $f_p(x)=(x-z_0)^{-p}$ with $z_0=a+ib$, with $-\infty<a<\infty$,
  $b>0$ and $p\in\mathbb{N}$. The magnitude of the remainder when
  using the $n$-point Gauss-Legendre rule to compute $\int_{-1}^1
  f_p(x)\d x$ is then in the asymptotic limit $n\to\infty$ given by
  \begin{align}
    |\oprem[f_p]| \simeq \frac{2\pi}{(p-1)!} 
    \left| \frac{2n+1}{\sqrt{z_0^2-1}} \right|^{p-1}
    \frac{1}{|z_0+\sqrt{z_0^2-1}|^{2n+1}}.
  \end{align}
  \label{thm:gl-complex}
\end{theorem}
\begin{proof}
  The proof follows from \eqref{eq:kn_gl},~\eqref{eq:knq} and
  \eqref{eq:rem_gl_complex_res}.
\end{proof}

We have (see e.g. \cite{Barrett1961}) that the factor $|z+\sqrt{z^2-1}|^{2n+1}$ is
constant on an ellipse with foci at $\pm 1$, and consequently that it has a minimum when
$\Re z = 0$. The decay rate for a general $z_0=a+ib$ is therefore bounded by the decay
rate given by $z_0=ib$,
\begin{align}
  \frac{1}{|z_0+\sqrt{z_0^2-1}|^{2n+1}} \le \frac{1}{|b+\sqrt{b^2+1}|^{2n+1}} .
\end{align}
Assuming $b \ll 1$ and discarding $\ordo{b^2}$ terms, we can simplify the result of
Theorem \ref{thm:gl-complex} to 
\begin{align}
  |\oprem[f_p]| \lesssim \frac{2\pi}{(p-1)!}(2n)^{p-1} e^{-2bn},
  \label{eq:gl_est_2d_simple}
\end{align}
where $a(n) \lesssim b(n)$ now denotes ''approximately less than or equal to'' in the
limit $n\to\infty$, in the sense that for a $K(n)$ such that $a(n) \le K(n)$, then
$\lim_{n\to\infty} K(n)/b(n) \approx 1$. Although just an approximation, the result in
\eqref{eq:gl_est_2d_simple} compares very well to numerical experiments, as shown in
Figure \ref{fig:gl_2d}.
\begin{figure}
  \centering
  \includetikz{0.48\textwidth}{0.5\textwidth}{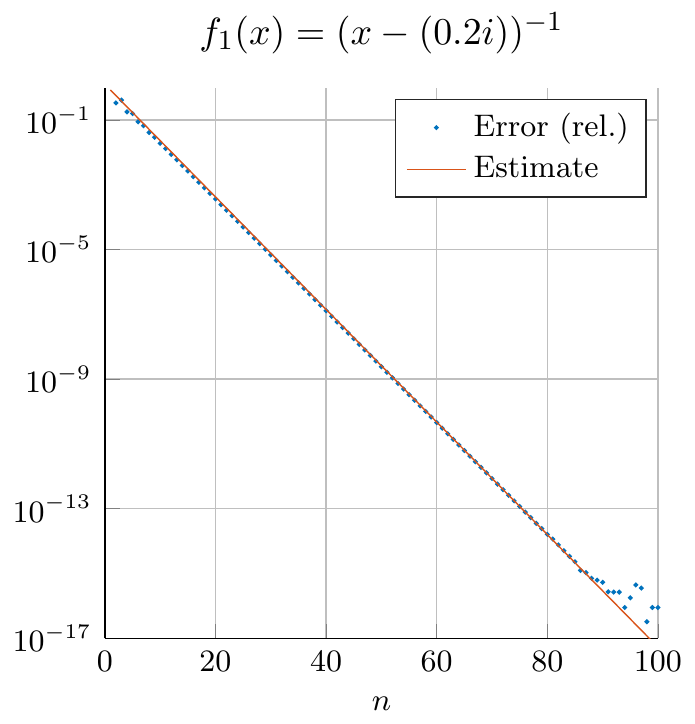}
  \includetikz{0.48\textwidth}{0.5\textwidth}{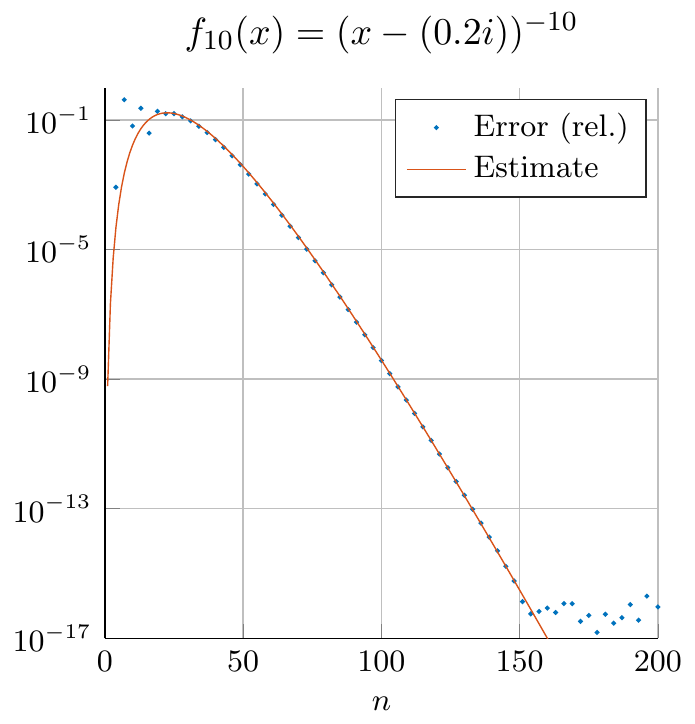}
  \caption{Gauss-Legendre rule quadrature error for $f_p(x)$ on $[-1,1]$, with estimate
    computed using \eqref{eq:gl_est_2d_simple}.}
  \label{fig:gl_2d}
\end{figure}

\subsubsection{Cartesian kernel}
\label{sec:gl-3d-kernel}
Let us now consider the Cartesian kernel \eqref{eq:gp_gl} and the quadrature error when
computing
\begin{align}
  \opint[g_p] = \int_{-1}^1 \frac{\d x}{((x-a)^2+b^2)^p}
\end{align}
using Gauss-Legendre quadrature. The case of $p=1$ was studied by Elliott, Johnston \&
Johnston \cite{Elliott2008}, and following their example we write the complex extension of
$g_p$ as
\begin{align}
  g_p(z) = \frac{1}{((z-a)^2+b^2)^p} = \frac{1}{(z-z_0)^p(z-\overline z_0)^p},
\end{align}
which has poles at $z_0$ and $\overline z_0$,
\begin{align}
  z_0 := a + ib.
\end{align}
Letting the contour $\contour$ enclose $[-1,1]$ and the poles, we see that the integrand
vanishes as we take $\contour$ to infinity, and we are left with a remainder determined by
the residues,
\begin{align}
  \oprem[g_p] = - \sum_{w=\{z_0,\overline z_0\}} \resbrac{
    \frac{k_n(z)}{(z-z_0)^p(z-\overline z_0)^p}, w } ,
  \label{eq:rem_gl_3d_def}
\end{align}
with $k_n$ as defined in \eqref{eq:kn_gl}. These residues are easily evaluated as
\begin{align}
  \resbrac{ \frac{k_n(z)}{(z-z_0)^p(z-\overline z_0)^p}, z_0 } 
  = \frac{1}{(p-1)!}\frac{\d^{p-1}}{\d z^{p-1}}
  \left( \frac{k_n(z)}{(z-\overline z_0)^p} \right)_{z=z_0}.
  \label{eq:def_Res(fp)}
\end{align}
Carrying out the full computations for
$p=1,2,3$, we get
\begin{align}
  \oprem[g_1] &= -\frac{1}{b}\Im[k_n(z_0)],  \\
  \oprem[g_2] &= -\frac{\Re \left[k_n'(z)\right]}{2 b^2} + \frac{\Im\left[k_n(z)\right]}{2 b^3}, \\
  \oprem[g_3] &= -\frac{ \Im[ k_n''(z_0)]}{8b^3} -\frac{3\Re[ k_n'(z_0)]}{8b^4} + \frac{3i\Im[ k_n(z_0)] }{8b^5}. \label{eq:gl_est_f3}
\end{align}
For large $n$, the dominating term will be the one with the highest $k_n$-derivative,
so we can make the approximation
\begin{align}
  \resbrac{ g_p(z)k_n(z),z_0 } \simeq \frac{1}{(p-1)!}\frac{k_n^{(p-1)}(z_0)}{(z_0-\overline
    z_0)^p} .
  \label{eq:Res(fp)}
\end{align}
We combine \eqref{eq:knq}, \eqref{eq:Res(fp)}, $z_0-\overline z_0=2ib$ and $k_n(\overline
z_0) = \overline{k_n(z_0)}$, to get
\begin{theorem}
  Let $g_p(x)=((x-a)^2+b^2)^{-p}$ with $-\infty<a<\infty$, $b>0$ and
  $p\in\mathbb{N}$, and let $\oprem[g_p]$ denote the remainder when
  using the $n$-point Gauss-Legendre rule to compute $\int_{-1}^1
  g_p(x)\d x$. Then, in the asymptotic limit $n\to\infty$,
  \begin{align}
    |\oprem[g_p]| \simeq
    \frac{2}{(p-1)!(2b)^p}
    \times
    \begin{cases}
      |\Im[k_n^{(p-1)}(z_0)]| &\text{if $p$ odd}, \\
      |\Re[k_n^{(p-1)}(z_0)]| &\text{if $p$ even},
    \end{cases}
    \label{eq:gl_est_final}
  \end{align}
  where $z_0=a+ib$ and
  \begin{align}
    k_n^{(q)}(z) \simeq \left( -\frac{2n+1}{\sqrt{z^2-1}} \right)^q
    \frac{2\pi}{(z+\sqrt{z^2-1})^{2n+1}}.
  \end{align}
  \label{thm:gl-cartesian}
\end{theorem}
\begin{proof}
  The proof follows from \eqref{eq:rem_gl_3d_def}, \eqref{eq:Res(fp)}
  and \eqref{eq:knq}.
\end{proof}
\begin{figure}
  \centering
  \includetikz{0.48\textwidth}{0.5\textwidth}{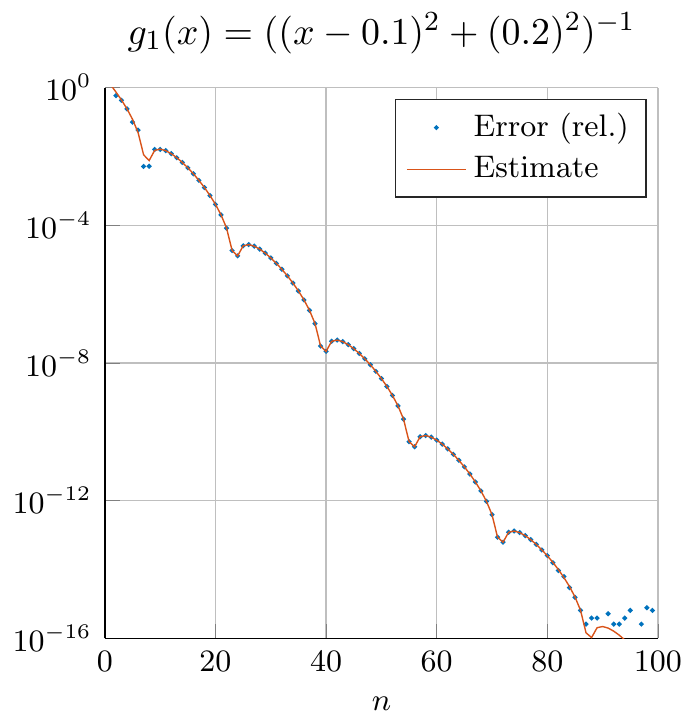}
  \includetikz{0.48\textwidth}{0.5\textwidth}{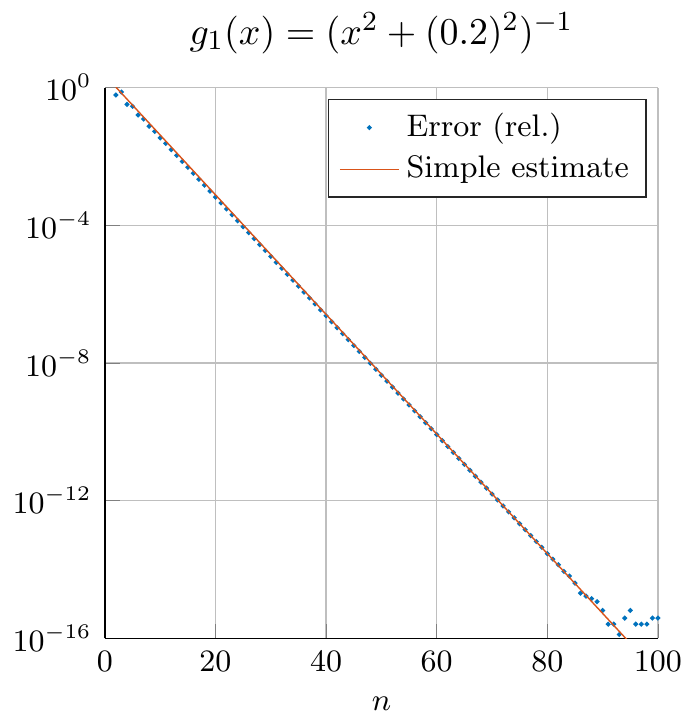}  \\
  \includetikz{0.48\textwidth}{0.5\textwidth}{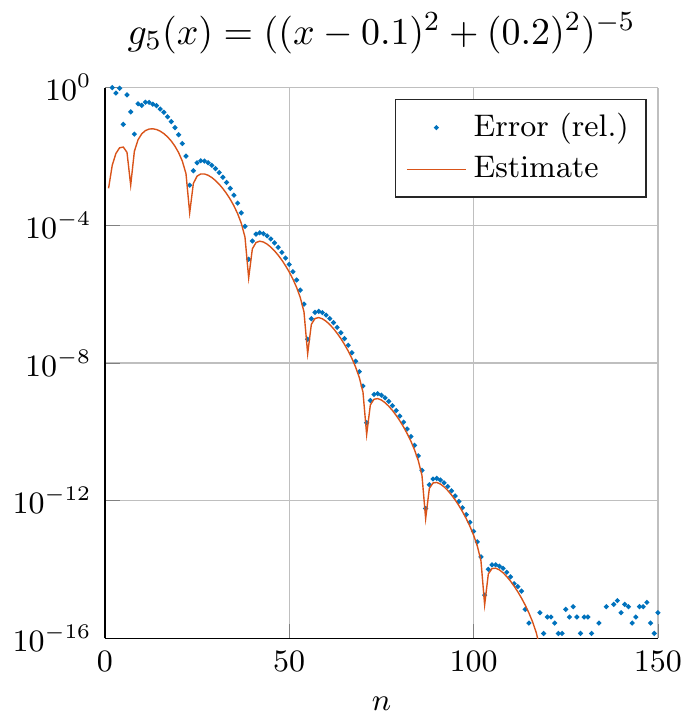}
  \includetikz{0.48\textwidth}{0.5\textwidth}{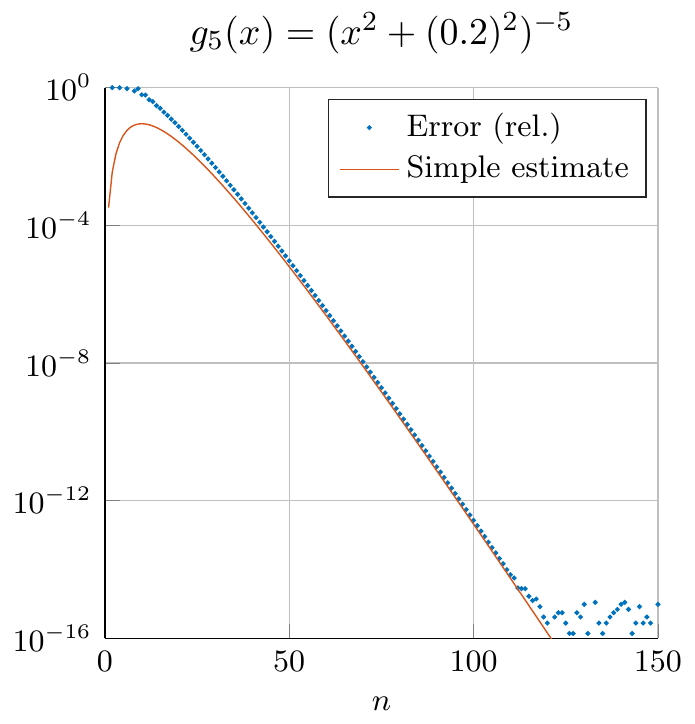}  \\
  \includetikz{0.48\textwidth}{0.5\textwidth}{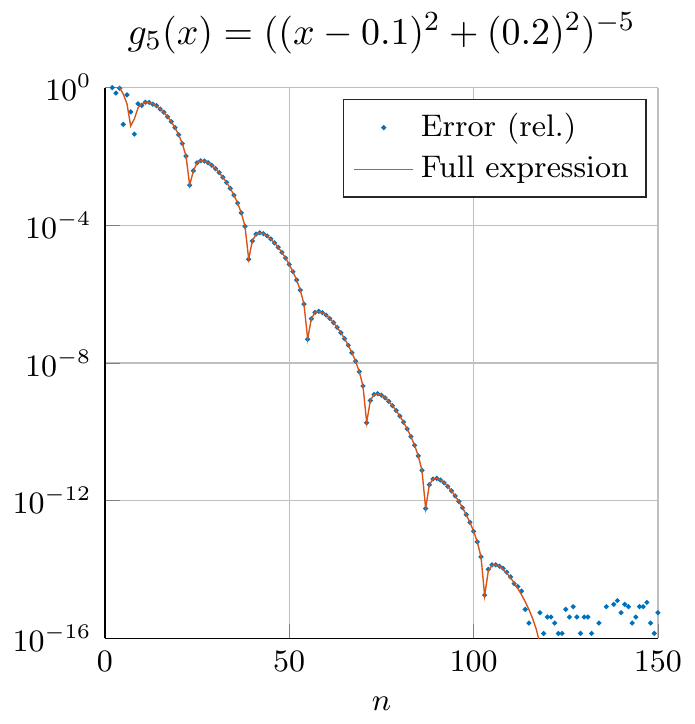}
  \includetikz{0.48\textwidth}{0.5\textwidth}{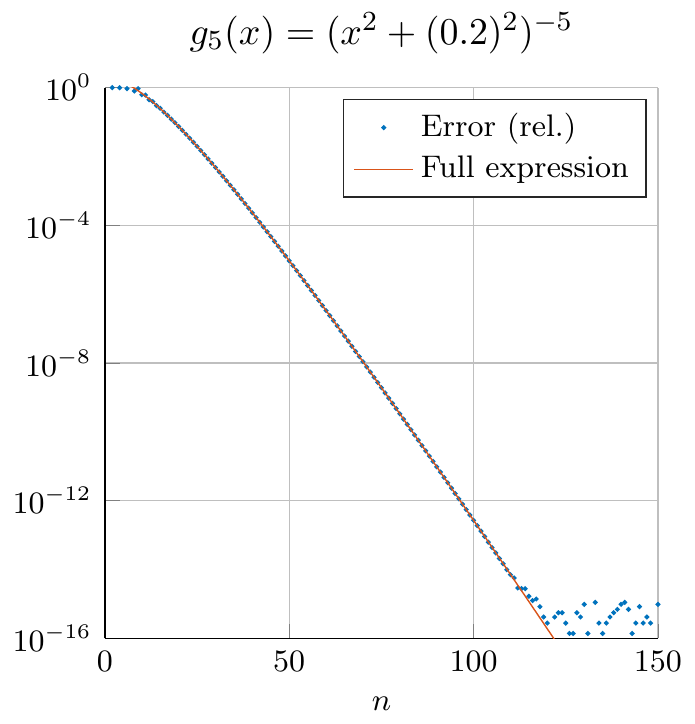}
  \caption{Gauss-Legendre quadrature error for $g_p(x)$ on $[-1,1]$, ''Full
    expression'' is computed using \eqref{eq:rem_gl_3d_def}--\eqref{eq:def_Res(fp)},
    ''Estimate'' using \eqref{eq:gl_est_final} and ''Simple estimate'' using
    \eqref{eq:gl_est_a0}.}
  \label{fig:gl_3d}
\end{figure}
Repeating the argument of section \ref{sec:gl-2d-kernel}, we can by assuming $b \ll 1$
estimate the largest error for all $a$ as
\begin{align}
  |\oprem[g_p]| \lesssim
  \frac{2\pi}{(p-1)! b^p}
  n^{p-1} e^{-2bn},
  \label{eq:gl_est_a0}
\end{align}
which provides a relatively clear view of how the error depends on $p$, $b$ and $n$.
Figure \ref{fig:gl_3d} shows experimental results using both this estimate and Theorem
\ref{thm:gl-cartesian}, as well as the results obtained when including all terms of the
differentiation in \eqref{eq:def_Res(fp)}. The left column of Figure \ref{fig:gl_3d} shows
results for $a \ne 0$, in which case the error has an oscillatory behavior in $n$. These
oscillations are bounded by the case $a=0$, shown to the right. The top row shows results
for $p=1$, where our estimates are very accurate. The center row shows that our estimates
lose accuracy for small $n$ at $p=5$, though that loss can be recovered by including all
derivatives, as in the bottom row.

\subsection{The trapezoidal rule}

For periodic integrands, the trapezoidal rule is typically the quadrature rule of
choice. It is easy to implement, has an even point distribution and converges
exponentially fast. In our analysis we will assume the interval of integration to be the
unit circle, parametrized in an arc length parameter $t$. Without loss of generality we
can then assume the singularity to lie on the $x$-axis at a distance $b$ from the
boundary,
\begin{align}
  z_0 &= 1+b, &
  (x_0,y_0) &= (1+b,0),
\end{align}
with $b > 0$. Our target kernels are then
\begin{align}
  f_p(z(t)) &= \frac{1}{(z(t)-z_0)^p}, 
  \label{eq:fp_trapz} \\
  g_p(x(t),y(t)) &= \frac{1}{((x(t)-x_0)^2 + y(t)^2)^p} 
  \label{eq:gp_trapz}.
\end{align}
The kernel $f_1$ corresponding to $p=1$ is essentially that of the
Laplace double layer potential. A bound on the quadrature error for
this potential was derived in \cite[Thm 2.3]{Barnett2014} for $\Gamma$
a general curve discretized using the trapezoidal rule. This bound
corresponds to the result in Theorem \ref{thm:trapz-complex} (below)
for $p=1$ and $\Gamma$ the unit circle.

\subsubsection{Contour integral}

We here state the necessary results for error estimation using contour
integrals for two cases: integral over a periodic interval and
integral over a circle in the complex plane. These results and more
can be found in the thorough review by Trefethen and Weideman
\cite{Trefethen2014}.  It is worth noting that the remainder functions
\eqref{eq:kn_trapz_per} and \eqref{eq:kn_trapz_circle} for the
trapezoidal rule are exact, in contrast to the asymptotic results used
in the Gauss-Legendre case.

\paragraph{Integral over a periodic interval}
For a $2\pi$-periodic function $f$ we approximate the integral
$\opint[f]=\int_0^{2\pi}f(x)\d x$ using the trapezoidal rule
\begin{align}
  \opquad[f] = \frac{2\pi}{n}\sum_{k=1}^n f(x_k), \quad x_k=2\pi k/n.
\end{align}
To compute $\oprem[f]$, we let $\contour$ be the rectangle $[0,2\pi] \pm ia$, $a > 0$,
traversed in the positive direction. The sides of rectangle cancel due to periodicity, so
we need only consider the top and bottom lines. The appropriate remainder function is
given by \cite{Trefethen2014}
\begin{align}
  k_n(z) = 2\pi i
  \begin{cases}
    \frac{-1}{e^{-inz}-1} & \Im z > 0,  \\
    \frac{1}{e^{inz}-1} & \Im z < 0 .
  \end{cases}
  \label{eq:kn_trapz_per}
\end{align}

\paragraph{Integral over circle in the complex plane}
For a function $f(z)$ on the unit circle we have $z=e^{it}$, such that
\begin{align}
 \opint[f] = \int_0^{2\pi} f(e^{it}) \d t
 ,%
\end{align}
and
\begin{align}
  \opquad[f] = \frac{2\pi}{n}\sum_{k=1}^n f(z_k), \quad z_k = e^{2\pi i k/n}.
\end{align}
For the contour integral we let $\contour$ be the circle $|z| = r > 1$, and the remainder
function is then given by \cite{Trefethen2014}
\begin{align}
  k_n(z) = \frac{-2\pi}{z(z^n-1)} .
  \label{eq:kn_trapz_circle}
\end{align}

\subsubsection{Complex kernel}
Integrating $f_p$ \eqref{eq:fp_trapz} on the unit circle in the complex plane, we wish to compute
\begin{align}
  \opint[f_p] = \int_0^{2\pi} \frac{\d t}{\left( z(t) - z_0 \right)^p},
\end{align}
with $z(t)=e^{it}$ and $z_0=1+b$. Letting $\contour$ enclose the unit circle and
$z_0$, we can compute the remainder of the trapezoidal rule as
\begin{align}
  \oprem[f_p] = \frac{1}{2\pi i} \int_\contour \frac{k_n(z)}{(z-z_0)^p} \d z
  - \resbrac{ \frac{k_n(z)}{(z-z_0)^p}, z_0 },
\end{align}
with $k_n$ as defined in \eqref{eq:kn_trapz_circle}. Taking $\contour$ to infinity the
integral vanishes, and the error is given by the residue at $z_0$, where there is a pole
of order $p$ such that
\begin{align}
  \oprem[f_p] = -\frac{1}{(p-1)!} k_n^{(p-1)}(1+b).
  \label{eq:rem_fp}
\end{align}
If we evaluate the derivative analytically, this expression is exact. For large $n$ we can
estimate $k_n$ as,
\begin{align}
  k_n(z) \simeq -2\pi z^{-(n+1)}
\end{align}
such that
\begin{align}
  k_n^{(p-1)} \simeq -2 \pi (-1)^{p-1} \frac{(n+1)\cdots(n+p-1)}{(p-1)!} z^{-(n+p)}.
  \label{eq:trapz_knq}
\end{align}
We can simplify the product in the numerator through
\begin{align}
  (n+1)\cdots(n+p-1) \simeq (n+p)^{p-1}.
  \label{eq:simple_prod}
\end{align}
Putting it all together, we get the following result:
\begin{theorem}
  Let $f_p(z)=(z-z_0)^{-p}$ with $p\in\mathbb{N}$, $|z_0|=1+b$ and $b>0$, and let
  $R_n[f_p]$ be the quadrature error when computing $\int_0^{2\pi}f_p(e^{it})\d t$ using
  the $n$-point trapezoidal rule. In the limit $n\to\infty$ we then have that
  \begin{align}
    |\oprem[f_p]| \simeq 2\pi\frac{(n+p)^{p-1}}{(p-1)!} (1+b)^{-(n+p)} .
    \label{eq:thm-trapz-complex}
  \end{align}  
  \label{thm:trapz-complex}
\end{theorem}
\begin{proof}
  The proof follows from \eqref{eq:rem_fp}, \eqref{eq:trapz_knq} and
  \eqref{eq:simple_prod}.
\end{proof}

The expression in Theorem \ref{thm:trapz-complex} gives a very good
estimate of the error. In Figure \ref{fig:trapz_2d} we compare it to
numerical results for $p=1$ and $p=10$, and in both cases it captures
the region of exponential convergence very well.

\begin{figure}
  \centering
  \includetikz{0.48\textwidth}{0.5\textwidth}{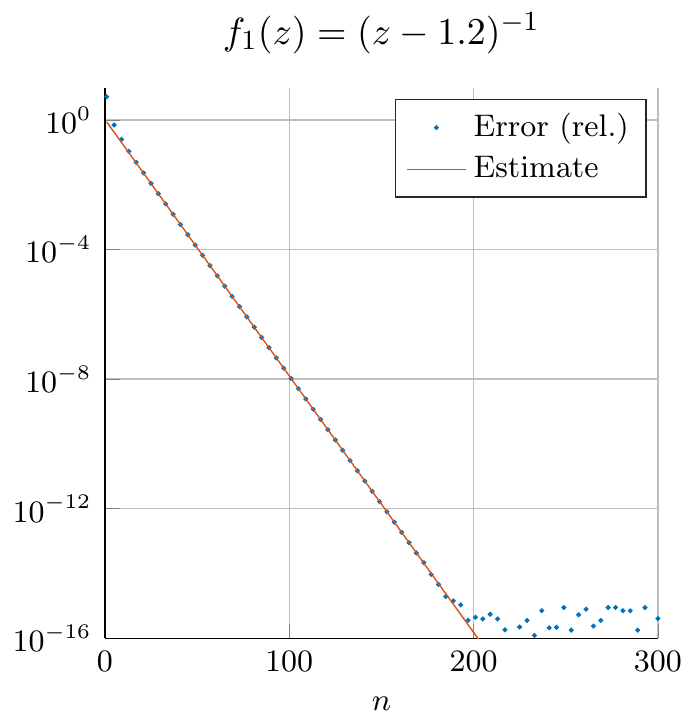}
  \includetikz{0.48\textwidth}{0.5\textwidth}{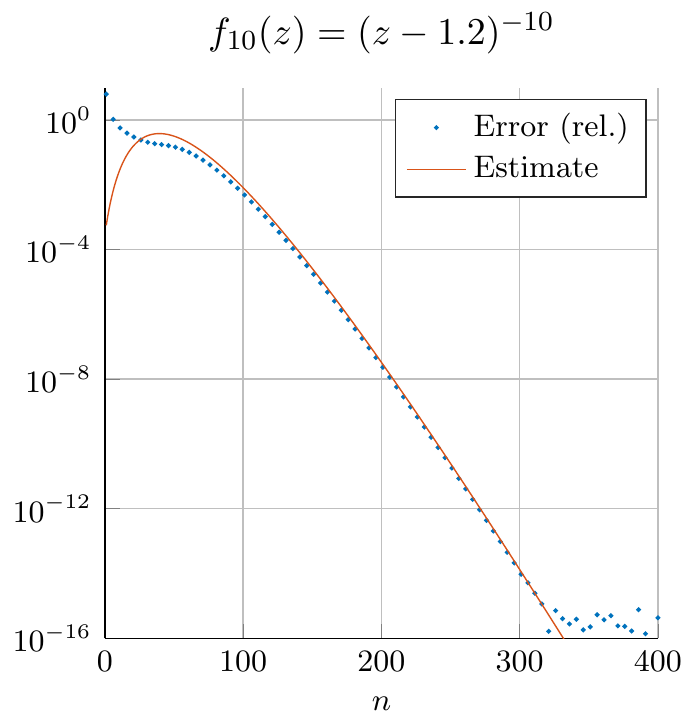}
  \caption{Trapezoidal rule quadrature error for $f_p(z)$ on the unit circle, compared to
    the estimate in Theorem \ref{thm:trapz-complex}.}
  \label{fig:trapz_2d}
\end{figure}

\subsubsection{Cartesian kernel}
\label{sec:trapz-cartesian-kernel}

Integrating $g_p$ \eqref{eq:gp_trapz} on the unit circle with $(x_0,y_0)=(1+b,0)$, our target integral is
\begin{align}
  \opint[g_p] = \int_0^{2\pi} \frac{\d t}{\left( (\cos t - x_0)^2 + \sin^2 t \right)^p},
\end{align}
where $t$ is an arc-length parameter describing the unit circle. Letting $g_p(z)$ be the
complex extension of $g_p(t)$, we can write
\begin{align}
  g_p(z) = \frac{1}{\left( 1 + x_0^2 - 2 x_0 \cos z \right)^p} .
\end{align}
This function has poles at $z_0$ and $\overline z_0$,
\begin{align}
  z_0 = i \log x_0 = i \log(1+b) .
  \label{eq:gl_trapz_z0}
\end{align}
It is also periodic on the interval $[0,2\pi]$, so we let $\contour$ be the rectangle
$[0,2\pi] \pm ia$, $a>0$, traversed in the positive direction, and take the remainder
function as defined in \eqref{eq:kn_trapz_per}. Letting $a$ go to infinity the
contribution from the contour vanishes, and the remainder is given by the residues at the
poles,
\begin{align}
  \oprem[g_p] = - \sum_{w=\{z_0,\overline z_0\}} \resbrac{
    g_p(z)k_n(z), w
  } .
  \label{eq:gp_trapz_rem}
\end{align}
To compute the residue at $z_0$ we begin with the definition \eqref{eq:def_residue}
\begin{align}
  \mathrm{Res}\left[ g_pk_n, z_0 \right] 
  = \frac{1}{(p-1)!} \lim_{z \to z_0} \frac{\d^{p-1}}{\d z^{p-1}}\left( 
    \frac{(z-z_0)^{p}}{(1 + x_0^2 - 2 x_0 \cos z)^p} k_n(z) 
  \right).
  \label{eq:gp_trapz_res}
\end{align}
Taking the limit for the first few $p$ we note that the residues from $z_0$ and $\overline
z_0$ are equal, such that we can express the remainder as
\begin{align}
  \oprem[g_1] &= \frac{2i k_n(z_0)}{b (b+2)}, \\
  \oprem[g_2] &= \frac{2 k_n'(z_0)}{b^2 (b+2)^2}+\frac{2 i (b (b+2)+2) k_n(z_0)}{b^3 (b+2)^3}, \\
  \oprem[g_3] &= -\frac{i k_n''(z_0)}{b^3 (b+2)^3} +\frac{3 (b (b+2)+2) k_n'(z_0)}{b^4
    (b+2)^4} + \mathcal O(k_n(z_0)) .
\end{align}
When $n$ is large we can estimate the remainder function as
\begin{align}
  k_n(z) \simeq 2\pi i
  \begin{cases}
    -e^{inz} & \Im z > 0,  \\
    e^{-inz} & \Im z < 0,
  \end{cases}
\end{align}
such that
\begin{align}
  k_n^{(q)}(z_0) \simeq -2\pi i (in)^q (1+b)^{-n} .
  \label{eq:gp_trapz_knq}
\end{align}
For large $n$ the remainder will be dominated by the highest derivative of $k_n$, so we can simplify to get the following result:

\begin{theorem}
  Let $g_p(x,y)=((x-x_0)^2+y^2)^{-p}$ with $p\in\mathbb{N}$, $x_0=1+b$ and $b>0$, and let
  $|\oprem[g_p]|$ be the quadrature error when computing $\int_0^{2\pi} g_p(\cos t,\sin
  t)\d t$ using the $n$-point trapezoidal rule. For $n\to\infty$ the error is then
  asymptotically given by
  \begin{align}
    |\oprem[g_p]| \simeq \frac{4\pi}{(p-1)!(b^2+2b)^p} \frac{n^{p-1}}{(1+b)^n} .
    \label{eq:thm-trapz-cartesian}
  \end{align}
  \label{thm:trapz-cartesian}
\end{theorem}
\begin{proof}
  The proof follows from \eqref{eq:gp_trapz_rem},
  \eqref{eq:gl_trapz_z0}, \eqref{eq:gp_trapz_res} and
  \eqref{eq:gp_trapz_knq}.
\end{proof}

\begin{figure}
  \centering
  \includetikz{0.48\textwidth}{0.5\textwidth}{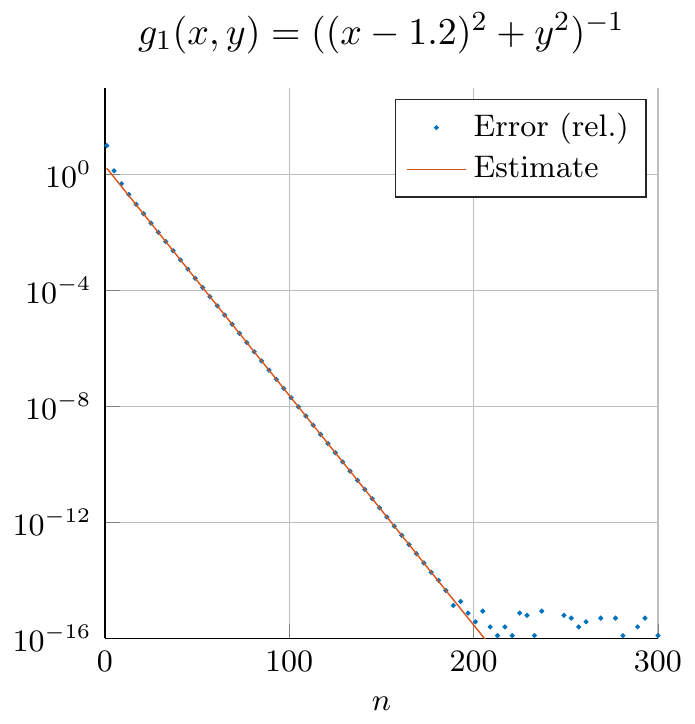}
  \includetikz{0.48\textwidth}{0.5\textwidth}{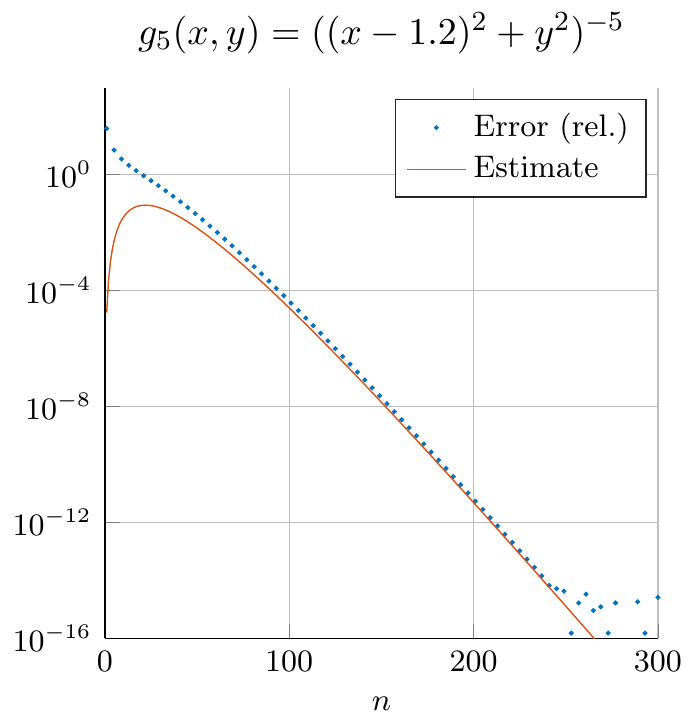}
  \caption{Trapezoidal rule quadrature error for $g_p(z)$ on the unit
    circle, with estimate given by Theorem \ref{thm:trapz-cartesian}.}
  \label{fig:trapz_3d}
\end{figure}

Figure \ref{fig:trapz_3d} shows the estimate of Theorem
\ref{thm:trapz-cartesian} applied for $p=1$ and $p=5$ with
$b=0.2$. The exponential convergence is well captured, though at low
$n$ and large $p$ some accuracy is sacrificed by our
simplifications. If the point $(x_0,y_0)$ is not on the $x$-axis, but
still on the circle of radius $1+b$, then the error has wiggles
similar to those present in Figure \ref{fig:gl_3d}. The result in
Theorem \ref{thm:trapz-cartesian} then follows the maximum of those
wiggles.

\subsection{Comments}
\label{sec:comments}

We now have asymptotically accurate estimates of $\oprem[f_p]$ and $\oprem[g_p]$ for the
trapezoidal and Gauss-Legendre rules on simple model geometries (unit circle and line
segment). The derivations of these four estimates all follow the same basic recipe, which
for an integrand $f$ and remainder function $k_n$ can be summarized as follows:
\begin{enumerate}[(i)]
\item Find the poles $\{z_i\}$ of $f$.
\item Take the error to be the residues of $k_nf$ at $\{z_i\}$.
\item If necessary, simplify the result by keeping only the term that dominates as
  $n\to\infty$. In the examples we have considered, this has been the one with the highest
  derivative of $k_n$.
\end{enumerate}
A few comments are also in order before we move on to applying our results to more general problems.

\paragraph{Non-integer poles} Our derivations of error estimates for $f_p$ and $g_p$ are
only valid for integer $p$. In practice we want to estimate the error for kernels like
$|\v x-\v x_0|^{-q}$, $q\in\mathbb N$, which means that $p$ also takes half-integer
values. Instead of carrying out new derivations for $p$ half-integer, we simply note that
the estimates we already have work well in practice also for half-integers, as long as the
factorial is computed using the gamma function,
\begin{align}
  (p-1)! = \Gamma(p),
  \label{eq:factorial_to_gamma}
\end{align}
which is true for integer $p$.

\paragraph{Taking the density into account}
When computing the QBX coefficients \eqref{eq:def_coeff_2d} and \eqref{eq:def_coeff_3d},
we typically have kernels like $f_p$ or $g_p$ times a non-constant density $\sigma$. For
the 2D coefficients $a_j$ in \eqref{eq:def_coeff_2d} we for example have the integrand
$\sigma(z)f_j(z)$. With a pole of order $j$ in $f_j$ at $z_0$ the residue contains all the
derivatives of $\sigma$ up to $\sigma^{(j-1)}$ evaluated at $z_0$, so a minimum
requirement is that those derivatives are bounded. Let us now assume $\sigma$ to be smooth
everywhere, which in the BIE setting is quite reasonable. The only residue to consider is
then the one at $z_0$, which is dominated by the highest derivative of $k_n$,
\begin{align}
  \resbrac{\sigma(z)f_j(z)k_n(z), z_0} \simeq \sigma(z_0) \resbrac{f_j(z)k_n(z), z_0} .
  \label{eq:2}
\end{align}
Depending on $\sigma$, the error may also have a contribution from the contour
$\mathcal C$, as the contour integral may be non-vanishing at infinity. This contribution
can however be safely neglected in the asymptotic region, as the contribution from the
poles then dominates the error, at least for poles close to the domain of integration.

As a concrete example, let us now consider the density $\sigma(x) = x^k e^{imx}$
multiplied with $g_p$, integrated on $[-1,1]$ using Gauss-Legendre quadrature. (A density
like this with $k$ or $m$ nonzero would appear when using a discretization that
corresponds to a polynomial or a Fourier series that is integrated term by term.) The
integrand
\begin{align}
  h(x) = \sigma(x)g_p(x) = \frac{x^k e^{imx}}{((x-a)^2+b^2)^p},
\end{align}
is analytic everywhere expect at the poles $z_0=a+ib$ and $\overline z_0$, so we can
compute the error as the contour integral of $k_nh$ over $\mathcal C$ plus the
residues. As $\mathcal C$ tends to infinity the integral does not vanish for any $n$,
since then $|k_n(z)h(z)| \simeq |z|^{k-2p-2n-1} e^{m|z|}$. Taking $\mathcal C$ to be a
circle of radius $R \gg 1$,
\begin{align}
  |\int_\mathcal{C} k_n(z)h(z)\d z| = \ordo{R^{k-2p-2n-1}e^{mR}} \le
  \ordo{R^{k-2n}e^{mR}}.
\end{align}
Minimizing the rightmost bound with respect to $R$ gives $R_{\text{min}}=\frac{2n-k}{m}$,
such that
\begin{align}
  |\int_\mathcal{C} k_n(z)h(z)\d z| \le \ordo{ \left(\frac{2n-k}{me}\right)^{k-2n} }
  \simeq \ordo{n^{-2n}},
\end{align}
We have $\sigma$ smooth everywhere, so we use the simplification \eqref{eq:2} of keeping
only the highest order derivative in $k_n$, such that
\begin{align}
  \oprem[h] &\simeq \sigma(z_0)\Res(g_pk_n,z_0) + \sigma(\overline z_0)\Res(g_pk_n,\overline
  z_0) + \mathcal{O}(n^{-2n})  .
\end{align}
The last term has faster than exponential decay, so the contribution from the poles will
dominate the error. Inserting \eqref{eq:Res(fp)},
\begin{align}
  |\oprem[h]| &\simeq \left(|\sigma(z_0)| + |\sigma(\overline z_0)|\right) \frac{|k_n^{(p-1)}(z_0)|}{(p-1)! (2b)^p} \\ 
  &= |z_0|^k (e^{mb} + e^{-mb})  \frac{|k_n^{(p-1)}(z_0)|}{(p-1)! (2b)^p}.
\end{align}
We can thus get an error estimate for $h$ by taking our previous results for the kernel
$f_p$, multiplied by the density $\sigma$ evaluated at the poles of $f_p$. 

Our experience is that the above reasoning holds true also in the general case, such that
we can estimate the quadrature error for an arbitrary, smooth density using the error
estimate for the kernel times the density evaluated at the poles. In practice we can
simplify this even further by using the values of the density on $\Gamma$, since that is
what we have access to in a numerical implementation. If $x_c$ is the point on $\Gamma$
closest to $z_0$, we then use that $\sigma(z_0) \approx \sigma(x_c)$ if $r$ small and
$\sigma$ smooth. For a nearly singular kernel $f$ this allows us to use existing error
estimates by writing
\begin{align}
  \oprem[\sigma f] \approx \sigma(x_c) \oprem[f] .
  \label{eq:est_with_sigma}
\end{align}

\section{Applications}
\label{sec:applications}

In the previous section we showed how to develop quadrature error estimates for the
kernels $f_p$ and $g_p$, when integrated on model geometries using the trapezoidal and
Gauss-Legendre quadrature rules. We will in this section show how these error estimates
can be used to estimate the quadrature errors of QBX. Before we do that, however, we will
show an example of how we can use our results to estimate the error when evaluating a
nearly singular layer potential using regular quadrature.

\subsection{Double layer potential in two dimensions}

To see how our results can be used when working with the double layer
potential, we now consider an example from Helsing \& Ojala
\cite[sec. 10.1]{Helsing2008}. We then consider the two-dimensional
double layer potential in complex form
\begin{align}
  u(z) = \int_\Gamma \sigma(w) \imbrac{\frac{\d w}{w-z}},
\end{align}
where $\sigma$ is the solution to an interior Dirichlet problem with a known reference
solution\footnote{We refer to \cite{Helsing2008} for details on $u_\text{ref}$ and how to
  compute $\sigma$.} $u_\text{ref}(z)$. The boundary is starfish-shaped with
parametrization
\begin{align}
  z(t) = (1 + 0.3 \cos 5t) e^{it}, \quad -\pi \le t \le \pi,
\end{align}
and is divided into 35 panels $\Gamma_i$ of equal length in $t$. Each panel is discretized
using a 16-point Gauss-Legendre quadrature rule, for a total of 560 discretization
points. Computing $u(z)$ in the interior using the Gauss-Legendre quadrature results in
large errors close to the boundary. This can be seen in Figure
\ref{fig:dbl_layer_starfish_errors}, which shows the relative pointwise error
\begin{align}
  e(z) = \frac{|u(z) - u_\text{ref}(z)|}{\|u_\text{ref}(z)\|_\infty} .
  \label{eq:dbl_layer_error}
\end{align}

An accurate estimate of $e(z)$ for $\Gamma$ discretized using the
trapezoidal rule can be observed in \cite[Fig. 1]{Barnett2014}. To
estimate $e(z)$ when using Gauss-Legendre panels, we can use our
results from section \ref{sec:gl-2d-kernel} to estimate the error
$e_i$ from each panel $\Gamma_i$, and then sum them together,
\begin{align}
  e(z) = \sum_{i=1}^{N_\text{panels}} e_i(z).
\end{align}
(In practice it suffices to use the contribution from the two closest panels, as the
closest panel completely dominates the error except for when $z$ is close to the edge
between two panels.) Replacing each panel with a corresponding flat panel, we can
generalize the results of Theorem \ref{thm:gl-complex} to estimate the error from panel
$\Gamma_i$ as
\begin{align}
  e_i(z) \lesssim \frac{2\pi \|\sigma\|_{L^\infty(\Gamma_i)}}{|z_0 + \sqrt{z_0^2-1}|^{2n+1}},
  \label{eq:dbl_layer_est}
\end{align}
where $z_0$ is the location of the pole $z$ under the transformation that takes $\Gamma_i$
to $[-1,1]$. We approximate the imaginary part of $z_0$ as $\Im z_0 = d/L$, where $L$ is
the length of $\Gamma_i$ and $d$ is the shortest distance from $z$ to $\Gamma_i$. The real
part $\Re z_0$ is approximated as the real part of $z$ after applying the scaling and
rotation that takes the endpoints of $\Gamma_i$ to $-1$ and $1$.

\begin{figure}[htbp]
  \centering
  \begin{subfigure}[t]{0.4\textwidth}
    \includegraphics[width=\textwidth]{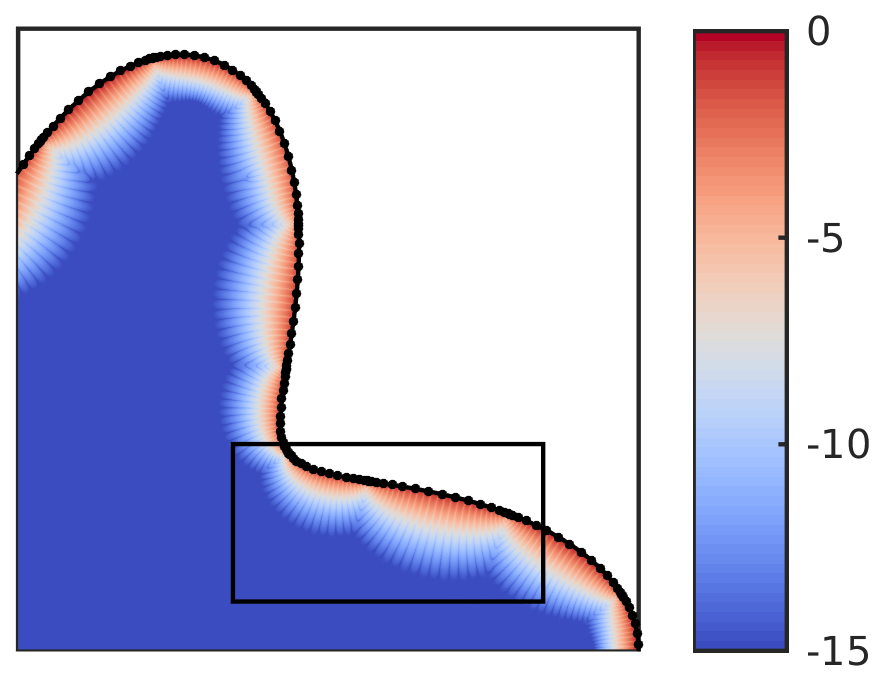}
    \caption{10-logarithm of the error $e(z)$ in a section of the starfish.}
    \label{fig:dbl_layer_starfish_errors}
  \end{subfigure}
  \hfill
  \begin{subfigure}[t]{0.59\textwidth}
    \includegraphics[width=\textwidth]{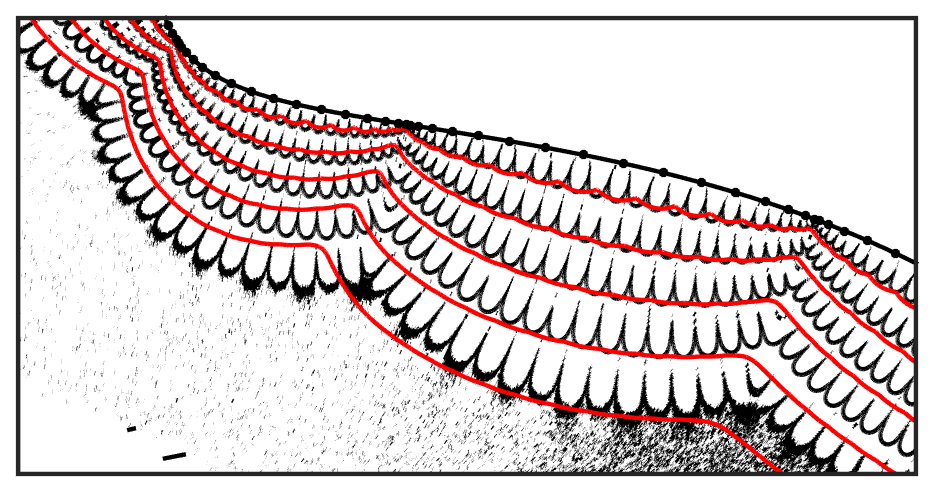}
    \caption{Level contours of $\log_{10}e(z) = \{-15$, $-12$, $-9$,
      $-6$, $-3\}$ in the cutout indicated in
      (\subref{fig:dbl_layer_starfish_errors}). Actual error
      \eqref{eq:dbl_layer_error} in \textbf{\color{black}black},
      estimate \eqref{eq:dbl_layer_est} in
      \textbf{\color{red}red}. The noise outside the contours are
      roundoff errors at $10^{-15}$.}
    \label{fig:dbl_layer_starfish_curves}
  \end{subfigure}
  \caption{Error when evaluating the double layer potential on a starfish domain using
    35 Gauss-Legendre panels with 16 points each.}
  \label{fig:dbl_layer_starfish}
\end{figure}

Evaluating the estimate \eqref{eq:dbl_layer_est} in the interior produces an error plot
which in ''eyeball norm'' is identical to that in Figure
\ref{fig:dbl_layer_starfish_errors}. If we are more careful and compare the level sets of
the error and the estimate (Figure \ref{fig:dbl_layer_starfish_curves}), we see that the
correspondence is extremely good for panels that are close to flat, while the estimate
suffers some inaccuracy for curved panels. Increasing the number of panels improves the
accuracy of the estimate (Figure \ref{fig:dbl_layer_starfish_n70}), as the individual
panels then are less curved. By removing the absolute value in the denominator of
\eqref{eq:dbl_layer_est} and instead taking the imaginary part of the whole expression we
can also reproduce the small-scale oscillations of the true error, though that has small
practical relevance.

\begin{figure}[htbp]
  \centering
  \begin{subfigure}{0.4\textwidth}
    \includegraphics[width=\textwidth]{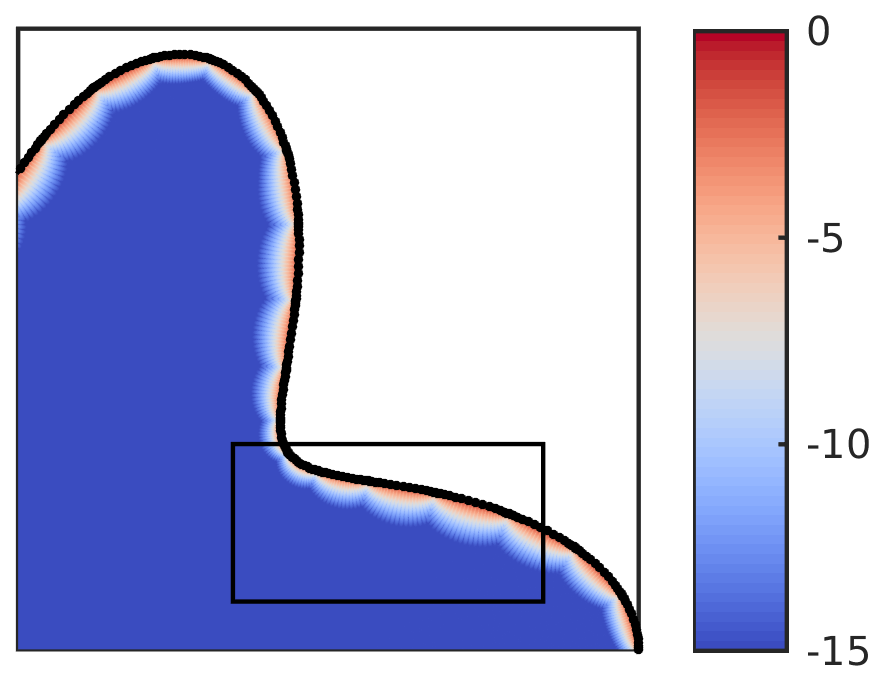}
    \caption{}
    \label{fig:dbl_layer_starfish_errors_n70}
  \end{subfigure}
  \hfill
  \begin{subfigure}{0.59\textwidth}
    \includegraphics[width=\textwidth]{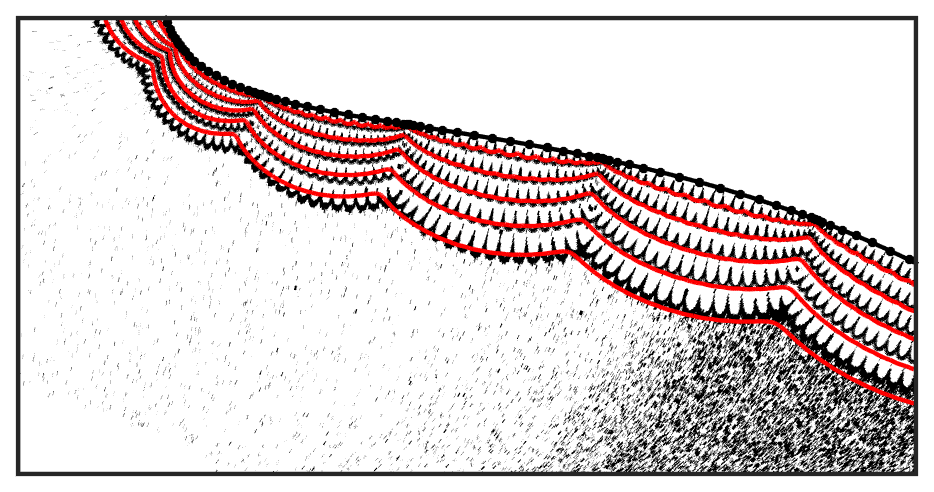}
    \caption{}
    \label{fig:dbl_layer_starfish_curves_n70}
  \end{subfigure}
  \caption{Same plots as in Figure \ref{fig:dbl_layer_starfish}, but with 70
    Gauss-Legendre panels.}
  \label{fig:dbl_layer_starfish_n70}
\end{figure}

These results suggest that our error estimates are useful for estimating the magnitude of
quadrature errors due to the near singularity of the integral. This allows for a cheap way
of determining when one needs to use a special quadrature method, such as QBX or the one
outlined in \cite{Helsing2008}.

\subsection{QBX in two dimensions}
Let us now return to the QBX quadrature error for the single layer potential in two
dimensions, which we introduced in section \ref{sec:quadr-expans-qbx}. We let $\Gamma$ be
a curve divided into $N$ panels of length $h$, and compute the coefficients $\tilde a_j$
at an expansion center $z_c$ from \eqref{eq:def_coeff_2d} using $n$-point Gauss-Legendre
quadrature on each panel. We then have from \eqref{eq:qbx_errors_2d} that the quadrature
error is
\begin{align}
  e_Q(z) = \Re \sum_{j=0}^p (a_j - \tilde a_j) (z-z_c)^j, \quad |z-z_c| \le r,
  \label{eq:2d_qbx_quad_err}
\end{align}
which is bounded by
\begin{align}
  |e_Q| \le \sum_{j=0}^p |a_j - \tilde a_j| r^j.
  \label{eq:2d_qbx_quad_err_simple}  
\end{align}
This error is discussed in Epstein et al. \cite{Epstein2013}, where they use the standard
Gauss-Legendre error estimate \eqref{eq:gl_classic} to get
\begin{align}
  |e_Q| \le C_n(p,\Gamma)\left(\frac{h}{4r}\right)^{2n} \| \varphi \|_{\mathcal C^{2n}} ,
\end{align}
from which it is concluded that $r > h/4$ is required for the quadrature to converge. This
does however not give any information about the rate at which the quadrature error grows
with $p$. Nor does it give any practically useful information about the quadrature error,
since it is easy to numerically verify that $r < h/4$ works just fine as long as $n$ is
large enough. In fact, the standard Gauss-Legendre error estimate is overly pessimistic
for this type of integrand, as discussed in section \ref{sec:class-error-estim}.

Using the result in \eqref{eq:gl_est_2d_simple}, as follows from
Theorem \ref{thm:gl-complex}, and generalizing as discussed in section
\ref{sec:comments}, we can estimate the quadrature error on the
interval $[-1,1]$ for a function of type $h(z) =
\sigma(z)(z-z_0)^{-j}$ as
\begin{align}
  |\oprem[h]| & \lesssim  |\sigma(z_0)| \frac{2\pi}{(j-1)!}(2n)^{j-1}e^{-2n \Im z_0},
\end{align}
under the assumption $\Im z_0 \ll 1$ and $\sigma$ smooth. Using a change of variables from
an interval of length $h$ to $[-1,1]$ and setting $z_0=2ir/h$, we can estimate the
quadrature error for a single coefficient as
\begin{align}
  |\tilde a_j - a_j| \lesssim 
  C(\Gamma,h)
  \frac{1}{j!} 
  \left(\frac{4n}{h}\right)^{j-1}
  e^{-4nr/h} \| \sigma \|_{L^{\infty}(\Gamma_r)},
  \label{eq:2d_qbx_simple_est}
\end{align}
where $\Gamma_r$ is the strip of width $r$ stretching from $\Gamma$ into the domain. In
practice we can use $\| \sigma \|_{L^{\infty}(\Gamma_r)} \approx \| \sigma
\|_{L^{\infty}(\Gamma)}$ if $r$ small and $\sigma$ smooth. For $\Gamma$ a perfectly flat
panel the constant would be $C(\Gamma,h)=2\pi$, and in practice $C$ is close to $2\pi$ if
the panels on a general $\Gamma$ are close to flat, since the error is dominated by that
from the panel closest to $z_c$. Putting \eqref{eq:2d_qbx_quad_err_simple} and
\eqref{eq:2d_qbx_simple_est} together allows us to estimate the quadrature error as
\begin{align}
  |e_Q| \lesssim C(\Gamma,h) \frac{h}{4n} 
  \left[ \sum_{j=0}^p \frac{1}{j!} \left(\frac{4nr}{h} \right)^j \right]
  e^{-4nr/h} \| \sigma \|_{L^{\infty}(\Gamma_r)}.
  \label{eq:2d_qbx_full_quad_err}
\end{align}
Discarding the term corresponding to $j=0$ (which is small) and using
Stirling's formula \eqref{eq:stirling}, we can simplify this to
\begin{align}
  |e_Q| \lesssim C'(\Gamma,h) \frac{h}{n} %
  \sum_{j=1}^p \frac{1}{\sqrt j} \left(\frac{4nre}{jh} \right)^j %
  e^{-4nr/h} \| \sigma \|_{L^{\infty}(\Gamma_r)} .
  \label{eq:2d_qbx_simpler_quad_err}
\end{align}
This expression gives a reasonably clear of view of how the error
depends on the involved variables, and Figure
\ref{fig:2dqbx_quad_error} shows that it captures the behavior of the
quadrature error quite well. The quotient $4r/h$ appears here too, but
without any type of bound; the convergence in $n$ will just stall as
$r/h \to 0$. Interestingly, the estimate
\eqref{eq:2d_qbx_simpler_quad_err} is similar to the bound on the
quadrature error derived in \cite[Thm 3.2]{Barnett2014} for the double
layer potential using the trapezoidal rule.

The sum over $j$ in \eqref{eq:2d_qbx_full_quad_err} and
\eqref{eq:2d_qbx_simpler_quad_err} makes the expressions a bit
cumbersome to interpret, though we can deduce that the error will in
some parameter regions grow exponentially in $p$. Recognizing that the
sum in \eqref{eq:2d_qbx_full_quad_err} is the truncated exponential
sum, we can write
\begin{align}
  \sum_{j=0}^p \frac{1}{j!} \left(\frac{4nr}{h} \right)^j \le e^{4nr/h},
\end{align}
such that, surprisingly, the quadrature error has an approximate upper bound independent
of $p$,
\begin{align}
  |e_Q| \lesssim C(\Gamma,h) \frac{h}{4n} \| \sigma \|_{L^\infty(\Gamma_r)} .
  \label{eq:2d_qbx_quad_err_ub}
\end{align}
Alternatively, we can use the definition of the incomplete gamma function $\Gamma(n,x)$
for integer $n$,
\begin{align}
  \Gamma(n+1,x) = n!e^{-x}\sum_{j=0}^{n}\frac{x^j}{j!}
  = \int_x^\infty t^{n}e^{-t} \d t,
\end{align}
to put the estimate of \eqref{eq:2d_qbx_full_quad_err} in a form that may not be easier to
interpret, but at least is more compact,
\begin{align}
  |e_Q| \lesssim C(\Gamma,h)
  \frac{h}{4np!}
  \Gamma(p+1, 4nr/h)
  \| \sigma \|_{L^{\infty}(\Gamma_r)} .
\end{align}

\begin{figure}
  \centering
  \tikzset{mark size=1}
  \includetikz{\textwidth}{0.4\textwidth}{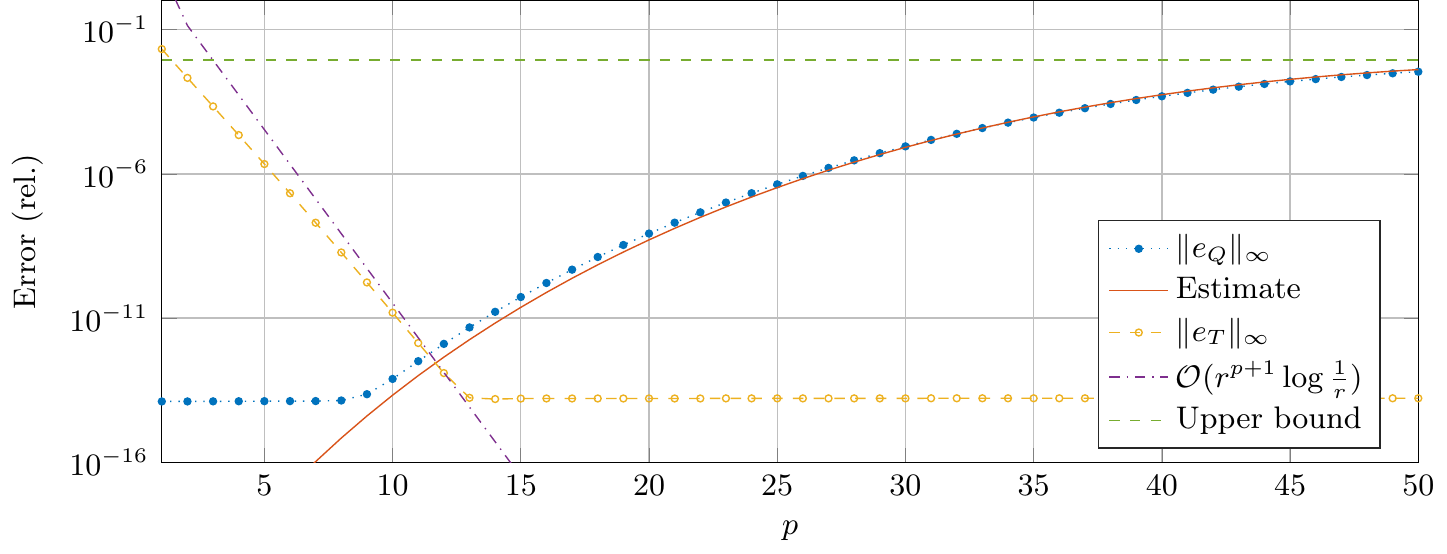}
  \caption{ Error components \eqref{eq:qbx_errors_2d} of 2D QBX on the
    unit circle, using density $\sigma = (\sin\theta)^{10}$,  measured at
    all nodes by comparing to a highly resolved reference
    solution. Numerical setup is $N=20$, $n=100$ and $r/h=0.1$. The
    quadrature error $e_Q$ and its upper bound are estimated using
    \protect\eqref{eq:2d_qbx_full_quad_err} and
    \protect\eqref{eq:2d_qbx_quad_err_ub} with $C=2\pi$, while the
    decay of the truncation error $e_T$ roughly follows
    \eqref{eq:trunc_err_bound_2d}.}
  \label{fig:2dqbx_quad_error}
\end{figure}

\subsection{QBX in three dimensions}

For the single layer potential in three dimensions, we need to
estimate the quadrature error in \eqref{eq:def_coeff_3d}. The surface
quadrature on $\Gamma$ is assumed to be a tensor product quadrature
rule, such that the surface is sliced into one-dimensional cross
sections. We will here show how to develop estimates for a simple
square Gauss-Legendre patch, but results for the more complicated case
of a spheroid can be found in appendix \ref{sec:qbx-spheroid}. Before doing
that, however, we need to work out two preliminaries: (i) How to
relate the remainder of a composite surface quadrature to the
remainders of the corresponding one-dimensional quadrature rules. (ii)
How to account for the magnitude and nearly singular behavior of the
spherical harmonics component of the kernel in
\eqref{eq:def_coeff_3d}.

\paragraph{Surface quadrature remainder}
Let $\opint^2$ denote a composite surface integral that is independent of the order of
integration, and let the subscript $\opint_s$ denote an integral carried out in the
variable $s$,
\begin{align}
  \opint^2[f(s,t)] := \opint_{t}\opint_{s}[f(s,t)] = \iint_\Gamma f(s,t) \d s\d t .
\end{align}
Applying a tensor product quadrature rule over the surface, we get
\begin{align}
  \opint^2f = \underbrace{ \opquadvar{t} \opquadvar{s} f }_{=: \opquad^2 f} +
  \opremvar{t}\opquadvar{s}f + \opquadvar{t}\opremvar{s}f + \opremvar{t}\opremvar{s}f.
\end{align}
Assuming the quadratic remainder term to be negligible and that $\opquadvar{\alpha}
\approx \opint_\alpha$, we can approximate the remainder of the surface quadrature as
\begin{align}
  \oprem^2f := (\opint^2 - \opquad^2)f \approx (\opint_s\opremvar{t} +
  \opint_t\opremvar{s})f.
  \label{eq:surface_quad_remainder}
\end{align}
So the surface remainder is approximately equal to the integrals of the one-dimensional
remainders, which is what one would expect. 

\newcommand{\spharmkernel}{K}
\newcommand{\legkernel}{K}
\paragraph{Spherical harmonics kernel}

For 3D QBX, our goal is to compute the quadrature error $e_Q$
\eqref{eq:qbx_errors_3d_quad}, 
\begin{align}
  e_Q &= \sum_{l=0}^p |\v x-\v x_0|^l \sum_{m=-l}^l (\alpha_l^m - \tilde
  \alpha_l^m) Y_l^{-m}(\theta_x,\varphi_x),
\end{align}
where
\begin{align}
  \alpha_l^m - \tilde\alpha_l^m =
  \frac{4\pi}{2l+1}
  \opremvar{\v y}^2\left[
  |\v y - \v x_0|^{-l-1}
  Y_l^{m}(\theta_y,\varphi_y)
  \sigma(\v y) \right]  .
\end{align}
We can simplify this by using the Legendre polynomial addition theorem
\cite{Greengard1997}, 
\begin{align}
  P_l(\cos\theta) = \frac{4\pi}{2l+1} \sum_{m=-l}^l
  Y_l^{-m}(\theta_x,\varphi_x) Y_l^{m}(\theta_y,\varphi_y),
\end{align}
where $\theta$ is the angle between $\v x$ and $\v y$. Using this, we
can simplify the quadrature error to
\begin{align}
  e_Q &= \sum_{l=0}^p |\v x-\v x_0|^l \opremvar{\v y}^2\left[ |\v y -
    \v x_0|^{-l-1} P_l(\cos\theta) \sigma(\v y) \right] .
  \label{eq:quad_err_simple}
\end{align}
Our task is then to estimate the quadrature error of the kernel
\begin{align}
  \legkernel_l(\v x, \v y) = \frac{ P_l(\cos\theta)}{|\v y - \v x_0|^{l+1}},
  \label{eq:legkernel}
\end{align}
which we shall refer to as the Legendre kernel. We are mainly
interested in the quadrature error when QBX is used for singular
integration, so we will assume $\v x \in \Gamma$, such that
\begin{align}
  \v x_0 -\v x = r\v{\hat n} \quad \text{ and } \quad \min_{\v
    y\in\Gamma} |\v y-\v x_0| = r.
\end{align}

In order to estimate $\oprem^2[\legkernel_l]$, let us now consider the
integral along a curve that is the intersection of $\Gamma$ and a
plane containing the expansion center $\v x_0=(x_0,y_0,z_0)$. We name
the curve $\gamma$ and assume without loss of generality that it lies
in the $xz$-plane, such that $\v y=(x,0,z)$ and $\v
x_0=(x_0,0,z_0)$. Then,
\begin{align}
  |\v y - \v x_0| = \sqrt{(x-x_0)^2 + (z-z_0)^2} \ge r .
\end{align}
Also assuming $\v x_0 - \v x = r\v{\hat z}$, we have that
\begin{align}
  \cos\theta &= \frac{z-z_0}{|\v y - \v x_0|}.
  \label{eq:costheta}
\end{align}

The Legendre polynomials $P_l(\cos\theta)$ are degree $l$ polynomials
in $\cos\theta$, and the quadrature error will be determined by the
leading coefficients of those polynomials, for reasons that will
become clear shortly. From Rodrigues' formula
\begin{align}
  P_l(x) = \frac{1}{2^ll!} \frac{\d^{l}}{\d x^{l}}(x^2-1)^l,
  \quad x \in [-1,1],
\end{align}
we can derive that 
\begin{align}
  P_l(\cos\theta) = \frac{(2l)!}{2^l(l!)^2} \cos^l\theta + 
  \mathcal O(\cos^{l-1}\theta).
\end{align}
Inserting \eqref{eq:costheta}, we get
\begin{align}
  P_l(\cos\theta) = \frac{(2l)!}{2^l(l!)^2} \frac{(z-z_0)^l}{|\v y - \v x_0|^l} +
  \ordo{\frac{(z-z_0)^{l-1}}{|\v y - \v x_0|^{l-1}}}.
\end{align}
Now we see the reason why the leading term in the polynomial will
dominate the quadrature error; $\cos\theta$ has a pole at $\v x_0$,
and the highest order of that pole will dominate the quadrature
error. The quadrature error for $\legkernel_l$ can therefore be found
by studying the quadrature error of the function
\newcommand{\spharmconst}{{B_l}}
\begin{align}
  \psi_l(x,z) &= \spharmconst
  \frac{(z-z_0)^l}{\left( (x-x_0)^2+(z-z_0)^2 \right)^{l+\frac 12}}, \\
  \spharmconst &= \frac{(2l)!}{2^l(l!)^2},
  \label{eq:spharmconst}
\end{align}
since on $\gamma$
\begin{align}
  \oprem[\legkernel_l(\v x, \cdot)] \approx \oprem[\psi_l].
  \label{eq:spharm_lt_psi}
\end{align}
The magnitude of the spherical harmonics $\spharmconst$ can be simplified using Stirling's
formula $n! \approx \sqrt{2\pi}n^{n+\frac 12}e^{-n}$,
\begin{align}
  \spharmconst \approx 
  \begin{cases}
    1, &\text{for } l=0, \\
    2^l / \sqrt{\pi l}, & \text{for } l \ge 1
  \end{cases}
  \label{eq:spharmconst_simple}
\end{align}
An error estimate for $\psi_l$ follows from our results for the
Cartesian kernel $g_p$ \eqref{eq:gp_def} in Theorems
\ref{thm:gl-cartesian} and \ref{thm:trapz-cartesian}, since
\begin{align}
  \psi_l(x,z) = \spharmconst (z-z_0)^l g_{l+\frac 12}(x,z) .
  \label{eq:psi}
\end{align}
In fact, a good estimate of the error is obtained by analyzing the
simpler form
\begin{align}
  \psi_l(x,z) = \spharmconst r^l g_{l+\frac 12}(x,z),
  \label{eq:psi_simple}
\end{align}
where $r=\min_{\v x\in\Gamma} |\v x-\v x_0|$.

\subsubsection{Gauss-Legendre patch}
\label{sec:gauss-legendre-patch}

\begin{figure}
  \centering
  \includegraphics{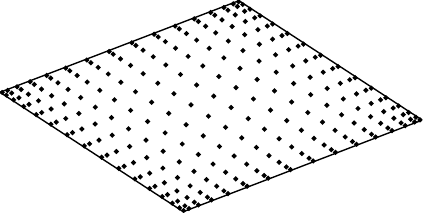}
  \caption{An $n \times n$ Gauss-Legendre patch.}
  \label{fig:gl_patch}
\end{figure}

When forming a quadrature rule for a general surface $\Gamma$, a straightforward method
that is often used in BIE methods is to divide the surface into approximately square
patches, and then use an $n\times n$ tensor product Gauss-Legendre rule on each patch
(which then looks like Figure \ref{fig:gl_patch}). Denoting the patches $\Gamma_i$, the
QBX expansion coefficients at a center $\v x_0$ are then computed as
\begin{align}
  \alpha_l^m = \frac{4\pi}{2l+1}
  \sum_i \int_{\Gamma_i}
  |\v y - \v x_0|^{-l-1}
  Y_l^{m}(\theta_y,\varphi_y)
  \sigma(\v y) \d S_{\v y},
  \label{eq:qbx_coeff_patches}
\end{align}
with the approximate coefficients $\tilde \alpha_l^m$ computed using
the associated quadrature rule for each patch. To be able to estimate
the QBX quadrature error $e_Q$ \eqref{eq:qbx_errors_3d_quad} we need
to be able to estimate the quadrature error from each patch when
evaluating \eqref{eq:qbx_coeff_patches} using Gauss-Legendre patches.

We now focus on the error from a single Gauss-Legendre patch, in the
special case of it being square and flat. For that we let $\Gamma$ be
the patch $(x,y,z) \in [-1,1]\times[-1,1]\times\{0\}$, and let $\v
x_0=(x_0,y_0,r)$ be a point close to $\Gamma$. We consider the
simplified form of the quadrature error \eqref{eq:quad_err_simple},
\begin{align}
  e_Q = \sum_{l=0}^p |\v x-\v x_0|^l \oprem^2\left[
    \legkernel_l(\v x, \cdot) \sigma(\cdot) \right],
 \label{eq:quad_err_simple_repeated}
\end{align}
where $\legkernel_l$ is the Legendre kernel \eqref{eq:legkernel}. To
estimate the quadrature error of $\legkernel_l$ on the patch, we
consider the quadrature error of the equivalent kernel $\psi_l^2$,
\begin{align}
  \psi_l^2(x,y) = \spharmconst r^l g^2_{l+\frac 12}(x,y),
  \label{eq:psi2}
\end{align}
which is the 2D analogue of \eqref{eq:psi_simple} and satisfies
$\oprem^2\legkernel_l\approx\oprem^2\psi_l^2$. Here $g_p^2$ is the
Cartesian kernel over the patch, defined as
\begin{align}
  g_p^2(x,y) = \frac{1}{((x-x_0)^2 + (y-y_0)^2 + r^2)^p} .
  \label{eq:def_gp2}
\end{align}
Evaluating the integral of $g_p^2$ on $\Gamma$ using the tensor product Gauss-Legendre
rule, we can expect (and verify) that the error for a given $r$ is largest for $\v x_0$
above the center of the patch, $x_0=y_0=0$. By symmetry we then have that
$\opint_x\opremvar{y}=\opint_y\opremvar{x}$, so we can use the results in
\eqref{eq:gl_est_a0} and \eqref{eq:factorial_to_gamma} with $b^2=y^2+r^2$ to estimate
$\opremvar{x}[g_p^2]$, which gives us
\begin{align}
  \oprem^2[g_p^2] \approx 2\opint_y\opremvar{x}[g_p^2] \lesssim 2 \int_{-1}^1 \frac{2\pi
    n^{p-1} }{\Gamma(p)(y^2+r^2)^{\frac p2}} e^{-2n\sqrt{y^2+r^2}} \d y .
  \label{eq:gl_patch_formulation}
\end{align}
We compute an approximation of this integral by expanding the square root around $y=0$,
\begin{align}
  \begin{split}
    \int_{-1}^1 (y^2+r^2)^{-p/2} e^{-2n\sqrt{y^2+r^2}} \d y &\approx r^{-p} \int_{-1}^1  e^{-2n(r + y^2/(2r))} \d y \\
    &= r^{-p}e^{-2nr}\sqrt{\pi r/n} \operatorname{erf}(\sqrt{n/r}) .
  \end{split}
  \label{eq:gl_patch_2d_integral}
\end{align}
We are considering large $n$ and small $r$, so $\operatorname{erf}(\sqrt{n/r})\approx
1$, such that
\begin{align}
  \int_{-1}^1 (y^2+r^2)^{-p/2} e^{-2n\sqrt{y^2+r^2}} \d y &\approx
  r^{-p} e^{-2nr} \sqrt{\pi r/n} .
  \label{eq:gl_patch_2d_integral_result}
\end{align}
This means that the result of the integration in $y$ is approximately
equal to multiplying the value of the integrand at $y=0$ with
$\delta=\sqrt{\pi r/n}$, independent of the integration bounds (as
long as the interval is wider than $\delta$). An interpretation of
this is that most of the error comes from a strip of width $\delta$
centered around $y=0$. Finally combining
\eqref{eq:gl_patch_2d_integral_result},
\eqref{eq:gl_patch_formulation} and \eqref{eq:psi2}, we get the
estimate for $\psi_l^2$ on a patch,
\begin{align}
  \oprem^2[\psi_l^2] \lesssim \frac{4\pi^{\frac 32}}{\Gamma(l+1/2)}
  n^{l-1} e^{-2nr} .
  \label{eq:est_psi2_patch}
\end{align}

We now consider a slightly more general case, where the patch is of
size $h \times h$. The corresponding change of variables from a unit
square patch in the integral of $\psi_l^2$ allows us to use the
results in \eqref{eq:est_psi2_patch} with the additional factor
$(2/h)^{l-1}$ and the change $r \to 2r/h$. It follows that the
quadrature error for the Legendre kernel $\legkernel_l$ on a flat
Gauss-Legendre patch with sides $h$ can be estimated as
\begin{align}
  \oprem^2[\legkernel_l] \lesssim 
  \frac{4\pi^{\frac 32} \spharmconst}{\Gamma(l+\frac 12)}
  \left( \frac{2n}{h} \right)^{l-1} e^{-4nr/h},
  \label{eq:est_qbxkernel_patch}
\end{align}
where $r$ is the distance from the expansion center to the patch.

We now let $\v x_t=(x_t,y_t,0)$, $-1 \le x_t,y_t \le 1$, be a target point on the patch
$\Gamma$ and $\v x_0=(x,y,r)$ be the corresponding expansion center. Using
\eqref{eq:est_with_sigma} gives us
\begin{align}
  \oprem^2[\legkernel_l\sigma] \approx 
  \sigma(\v x_t) \oprem^2[\legkernel_l] .
\end{align}
Inserting this, \eqref{eq:spharmconst} and
\eqref{eq:est_qbxkernel_patch} into
\eqref{eq:quad_err_simple_repeated} gives us the final expression for
the QBX quadrature error,
\begin{align}
  |e_Q(\v x_t)| &\lesssim |\sigma(\v x_t)| \frac{h}{n} \sum_{l=0}^p
  \frac{2\pi^{\frac 32}(2l)!}{\Gamma(l+\frac 12)(l!)^2} \left(
    \frac{nr}{h} \right)^{l} e^{-4nr/h} .
  \label{eq:est_quad_err_patch}
\end{align}
The accuracy of this estimate is demonstrated in Figure
\ref{fig:3dqbx_patch_full}, which shows the QBX errors for $\v x_t$ at
the center of the patch, $\v x_t=(0,0,0)$. The grid has $96 \times 96$
points, $r=0.2$ and $\sigma$ is a two-dimensional polynomial of degree
15 with random coefficients. This would correspond to using a 16th
order patch with factor 6 oversampling.
\begin{figure}
  \centering
  \tikzset{mark size=1}
  \includetikz{\textwidth}{0.4\textwidth}{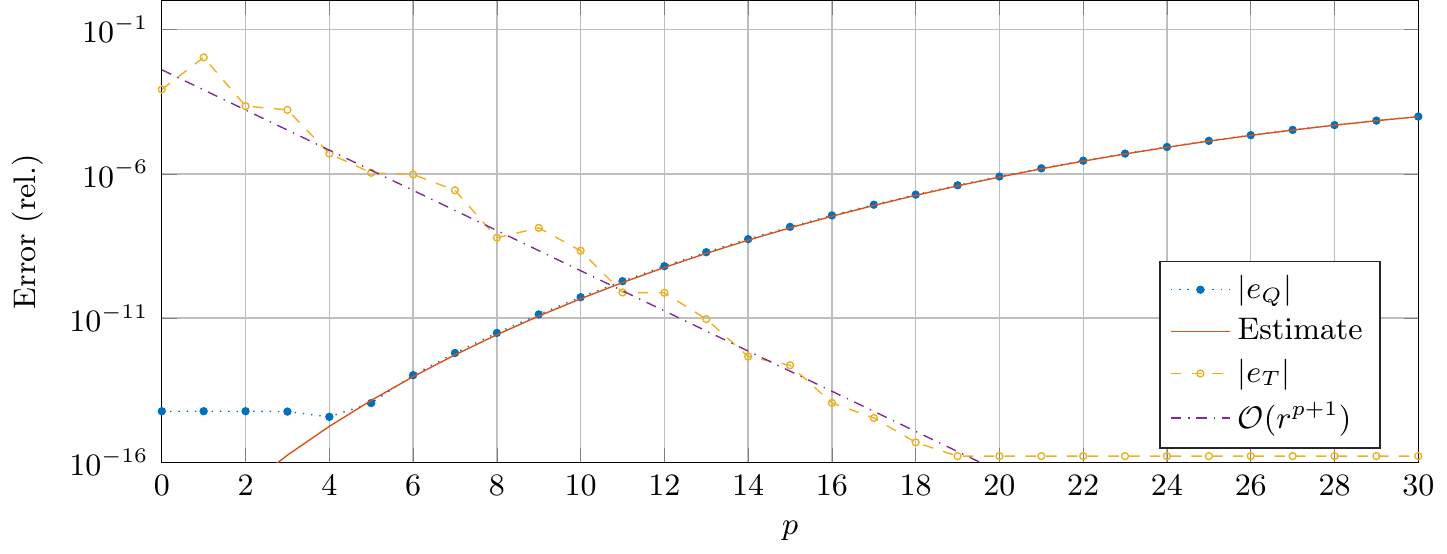}
  \caption{The error components \eqref{eq:qbx_errors_3d_trunc} and
    \eqref{eq:qbx_errors_3d_quad} of 3D QBX when applied to a
    Gauss-Legendre patch and compared against a reference
    solution. The quadrature error $e_Q$ is well approximated by the
    estimate \eqref{eq:est_quad_err_patch}, while the decay rate of
    the truncation error $e_T$ follows \eqref{eq:trunc_err_bound_3d}.}
  \label{fig:3dqbx_patch_full}
\end{figure}

To put our result in \eqref{eq:est_quad_err_patch} in a more general
form, we simplify it one step further. From Stirling's formula
\eqref{eq:stirling} we have that $\Gamma(l+1/2) \approx C l^l
e^{-l}$. Combining this with \eqref{eq:spharmconst_simple} and
discarding the $l=0$ term allows us to estimate the 3D QBX quadrature
error on a Gauss-Legendre patch as
\begin{align}
  |e_Q| \lesssim C \frac{h}{n} %
  \sum_{l=1}^p \frac{1}{\sqrt{l}} %
  \left( \frac{4nre}{hl} \right)^{l} %
  e^{-4nr/h} \|\sigma\|_{L^\infty(\Gamma)} .
  \label{eq:est_qbx_patch_simpler}
\end{align}
This is completely analogous to the 2D QBX result for Gauss-Legendre
panels \eqref{eq:2d_qbx_simpler_quad_err}.

\subsubsection{Complex geometries}

So far we have developed error estimates for simple geometries, but
the framework can be extended to more complex geometries in a
straightforward fashion, though the computations may be cumbersome. As
an example of how one can proceed, in appendix \ref{sec:qbx-spheroid} we
develop QBX quadrature error estimates for $\Gamma$ being a spheroid,
discretized using both the trapezoidal and Gauss-Legendre quadrature
rules. Another possibility would be to extend the previous section's
estimates for the Gauss-Legendre patch, to account for the patch
having curvature. This might be useful when estimating quadrature
errors on a general surface which has been divided into patches.

\subsection{Helmholtz kernel}
\renewcommand{\ell}{{l}}

As a final example, we shall briefly consider the use of QBX when
solving the Helmholtz equation $(\nabla+\omega^2)u=0$ in two
dimensions. This is the application which has been considered in the
majority of the QBX implementations to date
\cite{Barnett2014,Klockner2013,Rachh2016}. Jumping straight into the
details of the Helmholtz single layer potential (which can be found in
any of the above references), the quadrature error at an expansion
center $\v x_0$ is then given by
\begin{align}
  e_Q(\v x) = \sum_{l=-p}^p J_l(\omega |\v x-\v x_0|) e^{-il\theta_x} %
  (\alpha_l - \tilde\alpha_l).
  \label{eq:helmholtz_quad_error}
\end{align}
The expansion coefficients are computed as
\begin{equation}
  \alpha_\ell = \frac{i}{4} \int_{\Gamma} 
  H_\ell^{(1)}( \omega |\v y-\v x_0|) e^{i\ell\theta_y}
  \sigma(\v y) \d S_{\v y}, \quad \ell \in \{-p \dots p\},
  \label{eq:coeff_helmholtz}
\end{equation}
where $\theta_y$ is the polar angle of $\v y-\v x_0$. Here
$H_\ell^{(1)}$ is the Hankel function of the first kind and order
$\ell$, defined as
\begin{align}
  H_\ell^{(1)}(r) = J_\ell(r) + iY_\ell(r),
\end{align}
where $J_\ell$ is the Bessel function and $Y_\ell$ is the Neumann
function, both of order $\ell$. As $r \to 0$, $J_\ell$ goes smoothly
to zero, while $Y_\ell$ is singular. For our analysis based on residue
calculus, we need the leading order term of the singularity, which is
given by the power series of the Neumann function
\cite[\S10.8]{NIST:DLMF},
\begin{align}
  Y_\ell(r)=-\frac{2^\ell(\ell-1)!}{\pi}r^{-\ell}
   + \ordo{r^{-\ell+2}} , \quad \ell > 0.
   \label{eq:neumann_leading}
\end{align}

We now let $\Gamma$ be the interval $x \in [-1,1]$ and $\v x_0=(a,
b)$, $b>0$, such that $\v y=(x,0)$ and $|\v y-\v x_0| = \sqrt{(x-a)^2
  + b^2}$. We can then write $|\v y-\v x_0|=|x-z_0|$, where
$z_0=a+ib$. Using \eqref{eq:neumann_leading}, the Hankel function in
the integrand of \eqref{eq:coeff_helmholtz} can then be approximated
as
\begin{align}
  H_\ell^{(1)}(\omega |x-z_0|) \approx -i\frac{2^\ell(\ell-1)!}{\pi
    \omega^{\ell} } |x-z_0|^{-\ell},
\end{align}
since we know that the residue at $z_0$ will be dominated by the
highest order pole. We futher define $\theta$ to be the angle between
$x-z_0$ and $a-z_0$ in the complex plane. The exponential factor in
the integrand is then
\begin{align}
  e^{i\ell\theta} = \left( \cos\theta + i\sin\theta \right)^\ell  =
  \left( \frac{x-a+ib}{|x-z_0|} \right)^\ell 
  = \left( \frac{x-\overline z_0}{|x-z_0|} \right)^\ell,
\end{align}
such that
\begin{align}
  H_\ell^{(1)}(\omega |x-z_0|)e^{i\ell\theta} \approx
  -i\frac{2^\ell(\ell-1)!}{\pi
    \omega^{\ell} } \frac{(x-\overline z_0)^\ell}{|x-z_0|^{2\ell}} .
\end{align}
But
\begin{align}
  \frac{(x-\overline z_0)^\ell}{|x-z_0|^{2\ell}} = \frac{(x-\overline
    z_0)^\ell}{(x-z_0)^\ell(x-\overline z_0)^\ell} =
  \frac{1}{(x-z_0)^\ell} = f_\ell(x),
\end{align}
so the Helmholtz QBX kernel is in fact well approximated by the
complex kernel $f_\ell$ \eqref{eq:fp_gl} times a factor depending on
$\omega$ and $\ell$,
\begin{align}
  H_\ell^{(1)}(\omega |x-z_0|)e^{i\ell\theta} \approx
  -i\frac{2^\ell(\ell-1)!}{\pi \omega^{\ell} } f_\ell(x) .
\end{align}
Integrating this using the Gauss-Legendre rule, it follows directly
from Theorem \ref{thm:gl-complex} that
\begin{align}
  \left|\oprem\left[ H_\ell^{(1)}(\omega
      |x-z_0|)e^{i\ell\theta}\right]\right| \approx
  2 \left(\frac{2}{\omega}\right)^\ell
  \left| \frac{2n+1}{\sqrt{z_0^2-1}}
  \right|^{\ell-1} \frac{1}{|z_0+\sqrt{z_0^2-1}|^{2n+1}}.
\end{align}
We now let $\Gamma$ be a flat Gauss-Legendre panel of length $h$, and
$z_0$ be a point at a distance $r$ from $\Gamma$. Including the
variable change to $[-1,1]$ and following the simplifications in
\eqref{eq:gl_est_2d_simple} and \eqref{eq:est_with_sigma}, we can then
write
\begin{align}
  |\alpha_l - \tilde\alpha_l| = \frac{1}{4}%
  \left|\oprem\left[ H_\ell^{(1)}(\omega
      |x-z_0|)e^{i\ell\theta}\sigma\right]\right|  \lesssim
  \frac{h}{8n}
  \left(\frac{8n}{h\omega}\right)^\ell e^{-4nr/h}
  \|\sigma\|_{L^\infty(\Gamma)},
  \label{eq:helmholtz_coeff_est}
\end{align}
which is sharpest when $z_0$ lies on the line that extends normally
from the center of $\Gamma$. 

We are now ready to insert \eqref{eq:helmholtz_coeff_est} into
\eqref{eq:helmholtz_quad_error}. First, however, we assume that $|\v
x-\v x_0|=r$ and approximate the Bessel function using the first term
of its power series \cite[\S10.2]{NIST:DLMF},
\begin{align}
  J_l(\omega r) = \frac{1}{l!} \left(\frac{\omega r}{2}\right)^l +
  \ordo{%
    \frac{1}{(l+1)!} \left(\frac{\omega r}{2}\right)^{l+1}%
  }, \quad l>0 .
\end{align}
Discarding the negligible term corresponding to $l=0$, rewriting the
sum in \eqref{eq:helmholtz_quad_error} using positive $l$ (since
$|J_l|=|J_{-l}|$ and $|Y_l|=|Y_{-l}|$), and applying Stirling's
formula $l!\approx\sqrt{2\pi l}(l/e)^l$, we finally get the QBX
quadrature error for the Helmholtz kernel,
\begin{align}
  |e_Q| \lesssim \frac{1}{4\sqrt{2\pi}}\frac{h}{n}
  \sum_{l=1}^p %
  \frac{1}{\sqrt{l}}%
  \left(\frac{4nre}{hl}\right)^\ell e^{-4nr/h}
  \|\sigma\|_{L^\infty(\Gamma)} .
  \label{eq:est_helmholtz_final}
\end{align}
Remarkably, this estimate is identical (up to a constant) to that of
the Laplace single layer potential in both two
\eqref{eq:2d_qbx_simpler_quad_err} and three
\eqref{eq:est_qbx_patch_simpler} dimensions, even though the
underlying PDE is different. It works just as well as the previous
estimates, though the nature of our simplifications makes the accuracy
better for small $r$ and $\omega$. For $r/h = 1/2$ (which is quite
large), the estimate seems to have acceptable performance at least up
to $\omega h=20$.


\section{Conclusions}

The model kernels which we have considered, $f_p(z,w)=|z-w|^{-p}$ and
$g_p(\v x,\v y)=|\v x-\v y|^{-2p}$, can be found in two and three
dimensional BIE applications in general (for small $p$), and in QBX in
particular. Using the method of contour integrals, we have in section
\ref{sec:quadrature-errors} derived accurate estimates (Theorems
\ref{thm:gl-complex}--\ref{thm:trapz-cartesian}) of the quadrature
errors when integrating these kernels close to a singularity using the
$n$-point Gauss-Legendre and trapezoidal quadrature rules (on $[-1,1]$
and the unit circle, respectively). These estimates are not in the
form of bounds, which is what one classically seeks in numerical
analysis. Instead, they are asymptotic equalities valid in the limit
$n\to\infty$. Their key feature is that they predict the magnitude of
the errors surprisingly well also for small $n$, as we have seen
throughout this paper. Extracting the dominating features of our
estimates, we arrive at the summary presented in Table
\ref{tab:est_summary}, which shows how the quadrature errors depend on
$n$, $b$ and $p$.

\begin{table}[h]
  \centering
  \bgroup
  \def\arraystretch{1.5}
  \begin{tabular}[c]{l | c | c}
    $\ordo{ \|h\|_\infty^{-1} \oprem[h] (p-1)!}$ & Gauss-Legendre & Trapezoidal \\
    \hline
    Complex kernel $f_p$ & ${(2nb)^pe^{-2nb}}$ & ${(nb)^p(1+b)^{-n}}$ \\
    \hline
    Cartesian kernel $g_p$ & ${(nb)^pe^{-2nb}}$ & $(nb/2)^p (1+b)^{-n}$
  \end{tabular}
  \egroup
  \caption{Order of magnitude of relative error, scaled with common factor $(p-1)!$, for $h=f_p$ and $h=g_p$. Computed from estimates \eqref{eq:gl_est_2d_simple},~\eqref{eq:gl_est_a0},~\eqref{eq:thm-trapz-complex} and \eqref{eq:thm-trapz-cartesian}, using $\ordo{\|f_p\|_\infty}=b^{-p}$ and $\ordo{\|g_p\|_\infty}=b^{-2p}$, where $b$ is the distance from $\Gamma$ to the singularity.}
  \label{tab:est_summary}
\end{table}

By applying suitable generalizations, we have in section
\ref{sec:applications} of this paper demonstrated how to use our
results when working with BIEs and QBX. These results are
approximations rather than asymptotic equalities, but nevertheless
provide error estimates which are accurate enough to be used for
parameter selection in practical applications. Explicit estimates have
been developed for Gauss-Legendre panels in two dimensions, and for
flat Gauss-Legendre patches and spheroids (appendix
\ref{sec:qbx-spheroid}) in three dimensions. Though these results may
prove useful by themselves, the more important result is that the
methodology used to derive them can be generalized to other kernels
and geometries in a straightforward manner.

The QBX quadrature error $e_Q$ for the Laplace single layer potential
in two \eqref{eq:qbx_errors_2d} and three \eqref{eq:def_coeff_3d}
dimensions is well captured by our estimates
\eqref{eq:2d_qbx_simpler_quad_err} and
\eqref{eq:est_qbx_patch_simpler}, when using Gauss-Legendre as the
underlying quadrature. For the Helmholtz single layer potential in two
dimensions, the corresponding estimate is
\eqref{eq:est_helmholtz_final}. A common feature of these estimates is
that most of their dependence on the
parameters\footnote{Gauss-Legendre quadrature order $n$, expansion
  order $p$, expansion center distance $r$ and integration domain size
  $h$.}  $n$, $p$, $r$ and $h$ is captured by
\begin{align}
  |e_Q| \sim \frac{h}{n} \sum_{l=1}^p \frac{1}{\sqrt{l}} \left(
    \frac{4nre}{hl} \right)^{l} e^{-4nr/h}
  \|\sigma\|_{L^\infty(\Gamma)} .
\end{align}
This is in turn very similar to the results for the spheroid
\eqref{eq:qbx_spheroid_final_simpler} derived in appendix
\ref{sec:qbx-spheroid}, and to the bound for the double layer
potential in two dimensions derived by Barnett \cite[Thm
3.2]{Barnett2014}. Taken together, these expressions provide a better
understanding of the QBX quadrature error, which appears to have
similar behavior independent of kernel and dimension. Previous
results by Epstein et al. \cite{Epstein2013} for the truncation error
(\ref{eq:trunc_err_bound_2d},\ref{eq:trunc_err_bound_3d}) provide
insight into the truncation error and establish the analytic
foundation for the method. Putting these together with the results
presented here, it is now possible to understand the complete error
spectrum when working with QBX both in two and three dimensions.

We have in this paper focused on estimates for the harmonic single
layer potential in two and three dimensions. We have also shown how to
derive an estimate for the Helmholtz single layer potential in two
dimensions, and derivation of similar estimates for other kernels
should follow along the same lines.

\clearpage
\appendix
\begin{appendices}

\counterwithin{figure}{section}

\section{QBX quadrature error on spheroid}
\label{sec:qbx-spheroid}

\begin{figure}[h]
  \centering
  \includegraphics{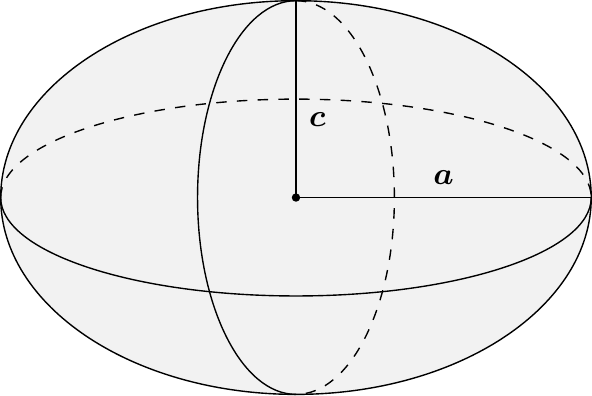}                
  \caption{A spheroid with semi-axes $a$ and $c$.}
  \label{fig:spheroid}
\end{figure}

We will now carry out essentially the same analysis as for the
Gauss-Legendre patch in section \ref{sec:gauss-legendre-patch}, but
for the specific case when the surface $\Gamma$ is a spheroid (see
Figure \ref{fig:spheroid}), defined as
\begin{align}
  \frac{x^2+y^2}{a^2} + \frac{z^2}{c^2} = 1.
\end{align}
The spheroid is denoted ''oblate'' when $a>c$, and ''prolate'' when $c>a$. Using a
parametrization $\{s\in[0,\pi], t\in[0,2\pi)\}$, we can describe the surface as
\begin{align}
  x &= a \cos s \sin t, \\
  y &= a \sin s \sin t, \\
  z &= c \cos t.
\end{align}
A straightforward quadrature for this surface is to use a tensor product quadrature with
the trapezoidal rule in the (periodic) $s$-direction and the Gauss-Legendre rule in the
$t$-direction. These two quadrature rules will then operate along cross sections that are
circles and half-ellipses, as shown in Figure \ref{fig:cross_sections}. 

To estimate the 3D QBX quadrature error as formulated in
\eqref{eq:quad_err_simple}, we will work with the Legendre kernel
$\legkernel_l$ \eqref{eq:legkernel} on the spheroid. Considering the
quadrature of $\legkernel_l$ \eqref{eq:legkernel} on $\Gamma$ with the
point $\v x_0$ at a distance $r$ away, we can expect the largest
quadrature errors to come from the two cross sections (a circle in $s$
and a half-ellipse in $t$) that are closest to $\v x_0$. To estimate
the total error $\oprem^2[\legkernel_l]$ on $\Gamma$ we first
need to estimate $\opremvar{s}\tzsup[\psi_l]$ and
$\opremvar{t}\glsup[\psi_l]$ on these cross sections, with $\psi_l$ as
defined in \eqref{eq:psi} and T and G denoting the trapezoidal and
Gauss-Legendre quadrature errors. Once we have those estimates, we can
approximate the integrals $\opint_t\opremvar{s}\tzsup[\psi_l]$ and
$\opint_s\opremvar{t}\glsup[\psi_l]$.  It then follows from the
properties of $\psi_l$ \eqref{eq:spharm_lt_psi} and our approximation
of the surface quadrature error \eqref{eq:surface_quad_remainder} that
\begin{align}
  |\oprem^2 [\legkernel_l]| \approx | \opint_t\opremvar{s}\tzsup [\psi_l] | + |
  \opint_s\opremvar{t}\glsup [\psi_l] | .
  \label{eq:est_3dqbx_spheroid}
\end{align}

\begin{figure}[htbp]
  \centering
  \begin{subfigure}[b]{6cm}
    \includegraphics{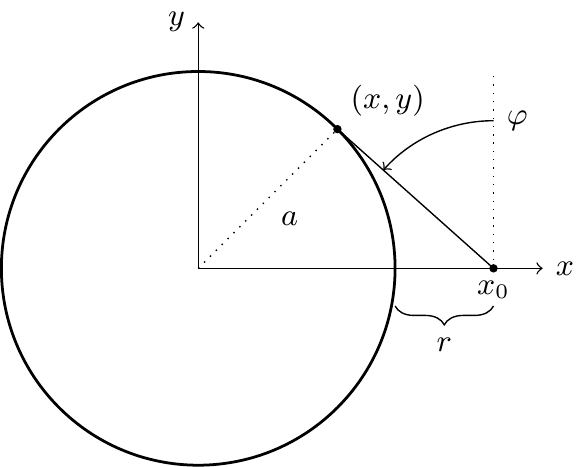}
    \caption{Circle at $z=0$, $\v x_0$ on the $x$-axis.}
    \label{fig:cross_section_circle}
  \end{subfigure}
  \hfill
  \begin{subfigure}[b]{5.5cm}
    \includegraphics{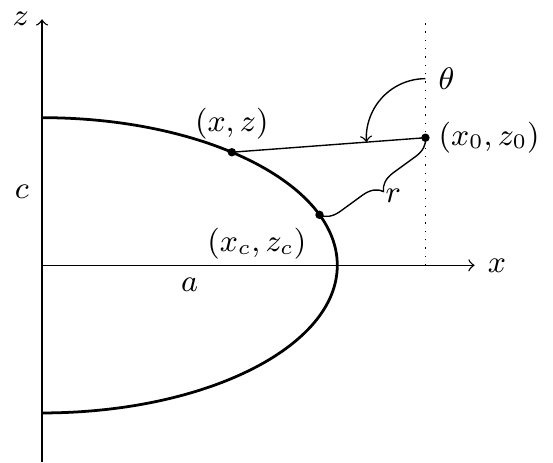}      
    \caption{Half-ellipse at $y=0$. }
    \label{fig:cross_section_ellipse}    
  \end{subfigure}
  \caption{Cross sections of a spheroid with $\v x_0$ at a distance $r$ away.}
  \label{fig:cross_sections}
\end{figure}

\paragraph{Trapezoidal rule on circular cross section.}
For the trapezoidal rule operating along the circular cross sections of the ellipse, we
can expect the largest quadrature error at an expansion center outside the middle
(equatorial) cross section, which has radius $a$ (Figure \ref{fig:cross_section_circle}),
since that is where the node distribution is sparsest. That cross section is described by
\begin{align}
  (x(s),y(s))=(a\cos s, a\sin s), \quad s \in [0,2\pi),
\end{align}
and we can without loss of generality set $\v x_0=(a+r,0,0)$, such that
\begin{align}
  \psi_l(s) = 
  \frac{\spharmconst (a\sin s)^l}{\left( (a\cos s-(a+r))^2+(a\sin s)^2 \right)^{l+\frac 12}} .
\end{align}
Rewriting the integral (where now $|\v x'(s)|=a$),
\begin{align}
  \int_0^{2\pi} \psi_l(s) a\d s
  =
  \frac{\spharmconst}{a^l} \int_0^{2\pi}
  \frac{\sin^ls \d s}{\left( (\cos s-(1+r/a))^2+\sin^2 s \right)^{l+\frac 12}},
\end{align}
we can use Theorem \ref{thm:trapz-cartesian} with $b=r/a$, $p=l+\frac{1}{2}$ and the
factorial generalization \eqref{eq:factorial_to_gamma} to estimate the quadrature error as
\begin{align}
  |\opremvar{s}\tzsup[\psi_l]| \simeq 
  \frac{4\pi a \spharmconst}{\Gamma(l+\frac 12) (2a)^l \sqrt{r(r+2a)}} 
  \frac{n^{l-\frac 12}}{(1+r/a)^{l+n}},
  \label{eq:est_3d_qbx_trapz_circle}
\end{align}
where we have used that the contribution from the numerator at the poles is
\begin{align}
  |\sin(\pm i \log(1+b))|^l = \left(\frac{b^2+2b}{2b+2}\right)^l .
\end{align}
To get the surface quadrature error we need to integrate this error across all cross
sections of the spheroid. Knowing that the quadrature error is local due to its fast
spatial decay, we simplify by integrating on the extension of the circular cross section
into an infinite cylinder of radius $a$, on which we integrate the estimate
\eqref{eq:est_3d_qbx_trapz_circle} after substituting $r \to \sqrt{z^2+r^2}$. We consider
only the factor $(1+ a^{-1} \sqrt{r^2 + z^2})^{-n}$, and expand the square root around
$z=0$. This leaves us with the integral
\begin{align}
  \int_{-\infty}^\infty \left(1+\frac{1}{a}\left(r
      + \frac{z^2}{2r} \right)\right)^{-n} \d z
  = \frac{\sqrt{2\pi r(a+r)}}{(1+r/a)^{n} }
  \frac{\Gamma(n-\frac 12)}{\Gamma(n)}.
\end{align}
For large $n$ we have
\begin{align}
  \frac{\Gamma(n-\frac 12)}{\Gamma(n)} \approx n^{-1/2},
\end{align}
so we can interpret the result as the quadrature error coming from a strip of width
$\sqrt{2\pi r(a+r)/n}$. Inserted into \eqref{eq:est_3d_qbx_trapz_circle}, this gives us
the desired estimate for the trapezoidal rule quadrature error on the sphere,
\begin{align}
  | \opint_t\opremvar{s}\tzsup[\psi_l] |
  &\approx \frac{4\pi a \spharmconst}{\Gamma(l+\frac 12) (2a)^l} 
  \sqrt{\frac{\pi (a+r)}{n(a+r/2)}}
  \frac{n^{l-\frac 12}}{(1+r/a)^{l+n}} .
\end{align}
We simplify this under the assumption $r \ll a$ to arrive at our final error estimate for
the trapezoidal rule error when integrating $\psi_l$ on a spheroid:
\begin{align}
  |\opint_t\opremvar{s}\tzsup[\psi_l]| \approx 
  \frac{\spharmconst}{\Gamma(l+\frac 12)}
  \frac{4\pi^{3/2} a}{n}
  \left( \frac{n}{2a} \right)^l
  \frac{1}{(1+r/a)^{l+n}}
  \label{eq:est_psi_circle}
\end{align}

\paragraph{Gauss-Legendre rule on half-elliptical cross section.}

When considering the Gauss-Legendre rule, we can without loss of generality limit
ourselves to the cross section where $x>0$ and $y=0$ (Figure
\ref{fig:cross_section_ellipse}), given by
\begin{align}
  \v x(t) = (x(t),z(t)) = (a\sin t,c\cos t), \quad s \in [0,\pi].
\end{align}
On this curve the outward (non-unit) normal is given by
\begin{align}
  \v n(t) = (c\sin t, a\cos t) .
\end{align}
Let $\v x_0 = (x_0, z_0)$ be a point at a distance $r$ from the curve, and let $(x_c,z_c)
= (x(t_c),z(t_c))$ be the point on the curve that is closest to $\v x_0$ (s.t. $|\v x_0 -
\v x_c|=r$). Then
\newcommand{\ncnorm}{|\v n(t_c)|}
\begin{align}
  x_0 &= \left( a
    + \frac{r c}{\ncnorm} \right) \sin t_c, \\
  z_0 &= \left(c \ + \frac{r a}{\ncnorm} \right) \cos t_c,
\end{align}
where $\ncnorm = \sqrt{a^2\cos^2t_c + c^2\sin^2t_c}$. We now want to estimate to
quadrature error for the integral
\begin{align}
  \opint[\psi_l] &= \int_0^\pi \psi_l(t) |\v x'(t)|\d t,
\end{align}
where
\begin{align}
  \psi_l(t) &= 
  \frac{\spharmconst (c\cos t)^l}{
    \left( (a\sin t-x_0)^2+(c\cos t - z_0)^2 \right)^{l+\frac 12}}, \\
  |\v x'(t)| &= \sqrt{a^2\cos^2t + c^2\sin^2t}.
\end{align}
Series expanding the denominator to second order around $t_c$, we get
\begin{align}
  \psi_l(t) &\approx
  \frac{\spharmconst (c\cos t - z_0)^l}{
  \left( k(t_c)^2(t-t_c)^2 + r^2\right)^{l+\frac 12} },
\end{align}
where
\begin{align}
  k(t_c) = \sqrt{\frac{acr + \ncnorm^3}{\ncnorm}}.
\end{align}
Changing variables from $t \in [0,\pi]$ to $u \in [-1,1]$, we can write the integral as
\begin{align}
  \opint[\psi_l] \approx \left( \frac{2}{\pi} \right)^{2l}\frac{\spharmconst}{k^{2l+1}}
  \int_{-1}^1 
  \frac{\left( c \cos t(u) - z_0 \right)^l |\v x'(t(u))|}
  {\left( (u-u_r)^2 + u_i^2 \right)^{l+\frac 12}}  \d u
\end{align}
where the integrand now has poles at $u_0$ and $\overline u_0$, 
\begin{align}
  u_0 &= u_r + i u_i, \\
  u_r &= (2 t_c - \pi)/ \pi,\\
  u_i &= \frac{2 r}{\pi k(t_c)}  .
\end{align}
The problem is now in the form considered in section \ref{sec:gl-3d-kernel}, such that we
can estimate the quadrature error. Evaluating the numerator at the poles $t_0 = t_c +
ir/k$, we get a contribution that we can approximate as
\begin{align}
  |c \cos(t_c+ir/k) - z_0|^l |\v x'(t_c+ir/k)| \approx 
  r^l |\v x'(t_c)| = 
  r^l \ncnorm .
\end{align}
Insertion into the quadrature error estimate \eqref{eq:gl_est_final} gives, after
some refactoring,
\begin{align}
  |\opremvar{t}\glsup[\psi_l]| \simeq
  \frac{\spharmconst}{\Gamma(l+\frac{1}{2})}
  \frac{ 
    2\pi^{3/2}\ncnorm 
  }{
    (\pi k)^l \sqrt{rk}
  }
  \left| \frac{2n+1}{\sqrt{u_0^2-1}} \right|^{l-\frac 12}
  \left| u_0+\sqrt{u_0^2-1} \right|^{-(2n+1)},
  \label{eq:half_ellipse_1}
\end{align}
where $u_0 = u_0(t_c)$ and $k=k(t_c)$.  This cumbersome expression accurately captures
both the order of magnitude and convergence rate of the error for all variations of the
geometrical quantities $a,c,r,t_c$. The rate of the exponential convergence in $n$ is
determined by the base
\begin{align}
  \beta(t_c) = \left| u_0+\sqrt{u_0^2-1} \right|^{-1} < 1 .
\end{align}
The closer the base is to unity, the slower the convergence. To find an upper bound of the
error estimate for a given geometry $a,c$ and distance $r$, we need to find the point
\begin{align}
  t^* = \operatornamewithlimits{argmax}_{t_c \in [0,\pi]} \beta(t_c) .
  \label{eq:tstar}
\end{align}
By studying the convergence rate for varying aspect ratios $a/c$ and $t_c\in[0,\pi]$, as
illustrated in Figure \ref{fig:ellipse_x0}, we can divide the problem of determining $t^*$
into two cases:
\begin{enumerate}[(i)]
\item $a \le c$ \\
  The base $\beta$ has a maximum at $t^*=\pi/2$, i.e. at the middle of the interval, and
  our variables then simplify to
  \begin{align}
    |\v n(t^*)| &= c,\\
    k &= \sqrt{c^2 + ar}, \\
    u_0 &= ib \\
    b &= 2 r/\pi\sqrt{c^2 + ar}.
  \end{align}
  Assuming $b$ small and discarding $\ordo{b^2}$ terms, we can simplify the
  resulting error estimate to 
    \begin{align}
      |\opremvar{t}\glsup[\psi_l]| &\lesssim
      \frac{ \spharmconst
      }{
        \Gamma(l+\frac 12)
      }
      \frac{\pi^2cb}{r^{3/2}}
      \left( \frac{nb}{r}\right)^{l-1/2} e^{-2bn}.
      \label{eq:half_ellipse_prol}
    \end{align}
    Also approximating $k \approx c$ (since $ar \ll c^2$), such that
    $b\approx 2r/\pi c$, results in a very compact expression:
    \begin{align}
     |\opremvar{t}\glsup[\psi_l]| &\lesssim
      \frac{ \spharmconst}{\Gamma(l+\frac 12)}
      \frac{2\pi}{\sqrt{r}}
      \left( \frac{2n}{\pi c}\right)^{l-\frac 12}
      e^{-4nr/\pi c} .
      \label{eq:half_ellipse_compact}
    \end{align}

  \item $a > c$ \\
    The base $\beta$ is symmetric about $\pi/2$ with two maxima of equal magnitude in the
    interval, as shown in Figure \ref{fig:ellipse_x0_map}. Here we have no closed form
    expression for $t^*$, so to find the worst case error estimate for a given $a,c,r$ we
    have to solve \eqref{eq:tstar} numerically and use the results in
    \eqref{eq:half_ellipse_1}.
\end{enumerate}

\begin{figure}[htbp]
  \centering
  \begin{subfigure}[b]{0.4\textwidth}
    \raisebox{0.2\textwidth}{
      \includetikz{\textwidth}{0.75\textwidth}{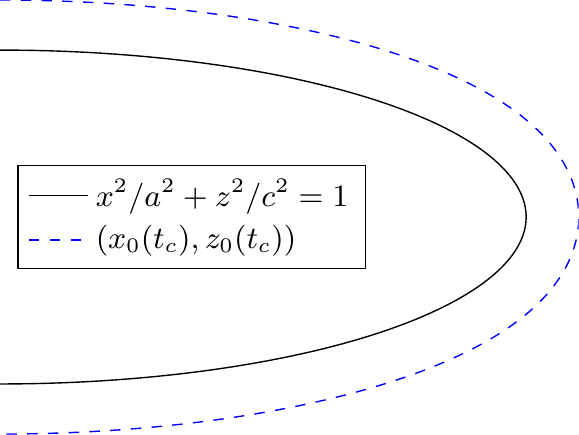}
    }
    \caption{Position of $\v x_0$ relative to the half-ellipse for $t_c\in[0,\pi]$.}
    \label{fig:ellipse_x0_dist}
  \end{subfigure}
  \hfill
  \begin{subfigure}[b]{0.5\textwidth}
      \includetikz{\textwidth}{0.8\textwidth}{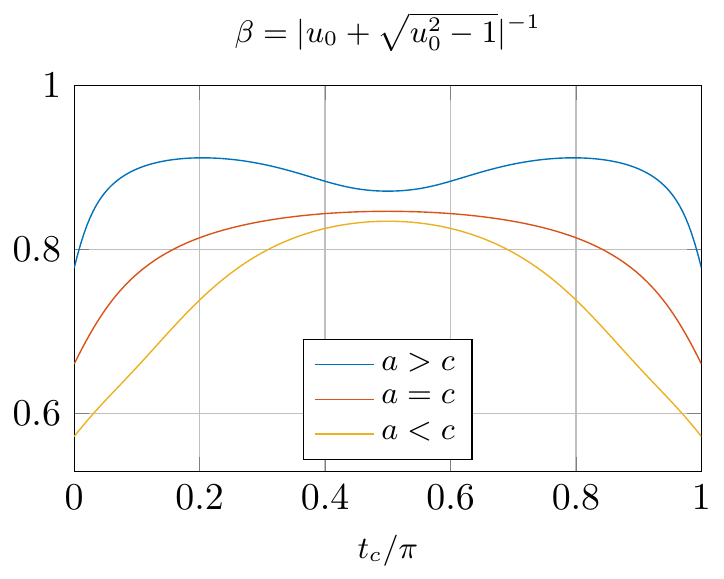}
      \caption{Base of exponential convergence as a function of $t_c$ for different aspect
        ratios.}
      \label{fig:ellipse_x0_map}
  \end{subfigure}
  \caption{}
  \label{fig:ellipse_x0}
\end{figure}

To get the error contribution from the entire surface, we would need
to integrate the estimate as the cross section of Figure
\ref{fig:cross_section_ellipse} is rotated around the $z$-axis. As an
approximation, we instead extend the cross section to infinity in the
positive and negative $y$-directions, and integrate the estimates
\eqref{eq:half_ellipse_1} and \eqref{eq:half_ellipse_compact} on the
resulting surface, with $r \to \sqrt{y^2+r^2}$. Repeating the process
used for the Gauss-Legendre patch in section
\ref{sec:gauss-legendre-patch} and for the circular cross section
above, it can be shown that the integrated error is approximately
equal to that from a strip of width $\pi\sqrt{\frac{r}{2n}
  \max(a,c)}$. Putting this together with the above results, we arrive
at the final estimate for the Gauss-Legendre error when integrating
$\psi_l$ on the spheroid:
\begin{align}
  |\opint_s \opremvar{t}\glsup[\psi_l]| \approx
  |\opremvar{t}\glsup[\psi_l]|
  \pi\sqrt{\frac{r}{2n} \max(a,c)},
  \label{eq:est_psi_hellipse}
\end{align}
where $\oprem\glsup[\psi_l]$ is estimated using
\eqref{eq:half_ellipse_prol} or \eqref{eq:half_ellipse_compact} for $a
\le c$ and \eqref{eq:half_ellipse_1} otherwise. Using
\eqref{eq:half_ellipse_compact} for $a \le c$, the final expression
can be written
\begin{align}
  |\opint_s \opremvar{t}\glsup[\psi_l]| \approx
  \frac{ \spharmconst}{\Gamma(l+\frac 12)}
   \frac{\pi^{5/2}c}{n}
  \left( \frac{2n}{c\pi}\right)^{l}
  e^{-4nr/c\pi} .
  \label{eq:est_psi_hellipse_compact}
\end{align}

\paragraph{Total error}

\begin{figure}[htbp]
  \centering
  \tikzset{mark size=1}
  \begin{subfigure}[b]{0.49\textwidth}
    \includetikz{\textwidth}{\textwidth}{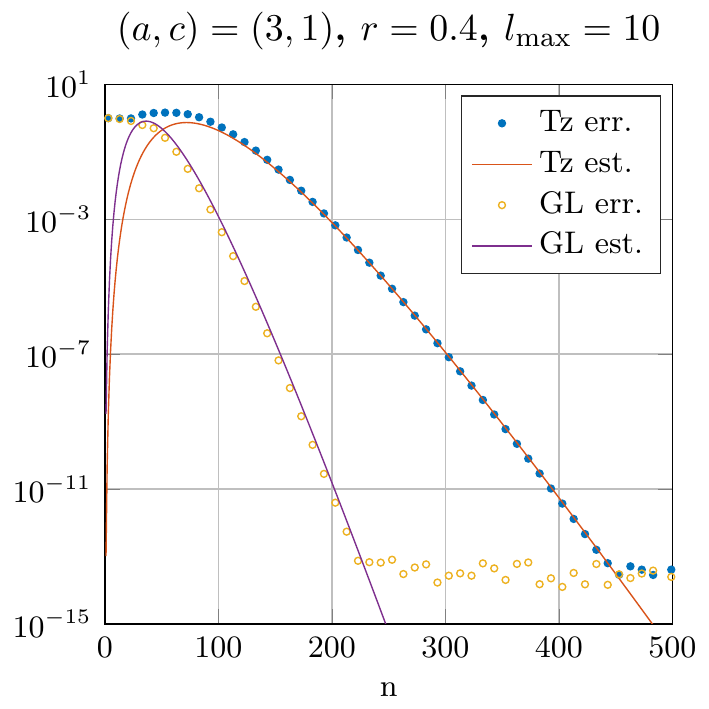}
    \caption{Oblate spheroid}
  \end{subfigure}
  \hfill
  \begin{subfigure}[b]{0.49\textwidth}
    \includetikz{\textwidth}{\textwidth}{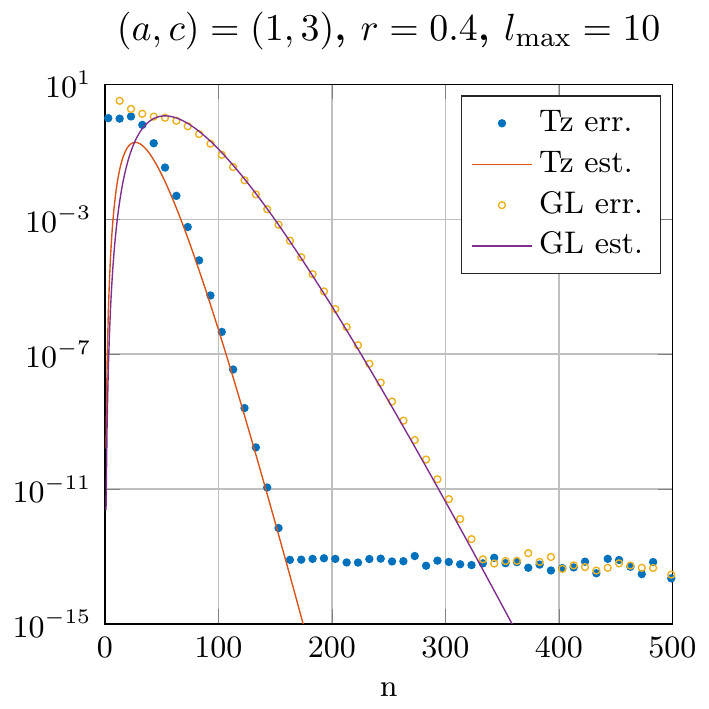}
    \caption{Prolate spheroid}
  \end{subfigure}
  \caption{Legendre kernel quadrature error $\oprem^2 [\legkernel_l]$
    for the $n$-point Gauss-Legendre (GL) and trapezoidal (Tz) rules
    on one oblate and one prolate spheroid, when the quadrature in the
    other direction is fully resolved. Estimates computed using
    \eqref{eq:est_psi_circle}, \eqref{eq:est_psi_hellipse} and
    \eqref{eq:est_psi_hellipse_compact}.}
  \label{fig:spheroid_conv}
\end{figure}

To test the estimates for the Legendre kernel $\legkernel_l$ on a
spheroid, we compute the error estimate \eqref{eq:est_3dqbx_spheroid}
using \eqref{eq:est_psi_circle}, \eqref{eq:est_psi_hellipse} and
\eqref{eq:est_psi_hellipse_compact}. To isolate the Gauss-Legendre and
trapezoidal rule errors, we vary the number of nodes $n$ in one
direction at a time, while the number of nodes in the other direction
is set so large that the quadrature in that direction is fully
resolved. The results, computed for one oblate and one prolate
spheroid up to $l=10$ and shown in Figure \ref{fig:spheroid_conv},
confirm that our estimates are accurate.

Putting the results together in the same way as for the Gauss-Legendre
patch of section \ref{sec:gauss-legendre-patch}, the estimate for the
QBX quadrature error on a spheroid becomes
\begin{align}
\begin{split}
  |e_Q(\v x_t)| \lesssim |\sigma(\v x_t)| \sum_{l=0}^p 
  r^l 
  \left( 
    | \opint_t\opremvar{s}\tzsup [\psi_l] | + |
    \opint_s\opremvar{t}\glsup [\psi_l] |
  \right) .
\end{split}
\label{eq:qbx_spheroid_final}
\end{align}
An example of this estimate is shown in Figure
\ref{fig:3dqbx_spheroid_full}. Using the result
\begin{align}
  \frac{\spharmconst}{\Gamma(l+1/2)} \approx C l^{-1/2}
  \left(\frac{2e}{l}\right)^l,
\end{align}
and applying the (quite crude) approximation $(1+r/a)^{n+l} \approx
e^{nr/a}$ to the trapezoidal rule results, we can see that
\eqref{eq:qbx_spheroid_final} for $a \le c$ behaves like
\begin{align}
\begin{split}
  |e_Q| \lesssim C
    \sum_{l=1}^p  \Bigg[ &
    \frac{a}{n\sqrt l} \left(\frac{nre}{al}\right)^l e^{-nr/a} 
    +  \frac{c}{n\sqrt l} \left(\frac{4nre}{\pi cl}\right)^l e^{-4nr/c\pi}
  \Bigg] \|\sigma\|_{l^\infty(\Gamma)}.
  \label{eq:qbx_spheroid_final_simpler}
\end{split}
\end{align}
This is similar to the results for the Gauss-Legendre patch
\eqref{eq:est_qbx_patch_simpler}.

\begin{figure}[htbp]
  \centering
  \tikzset{mark size=1}
  \includetikz{\textwidth}{0.4\textwidth}{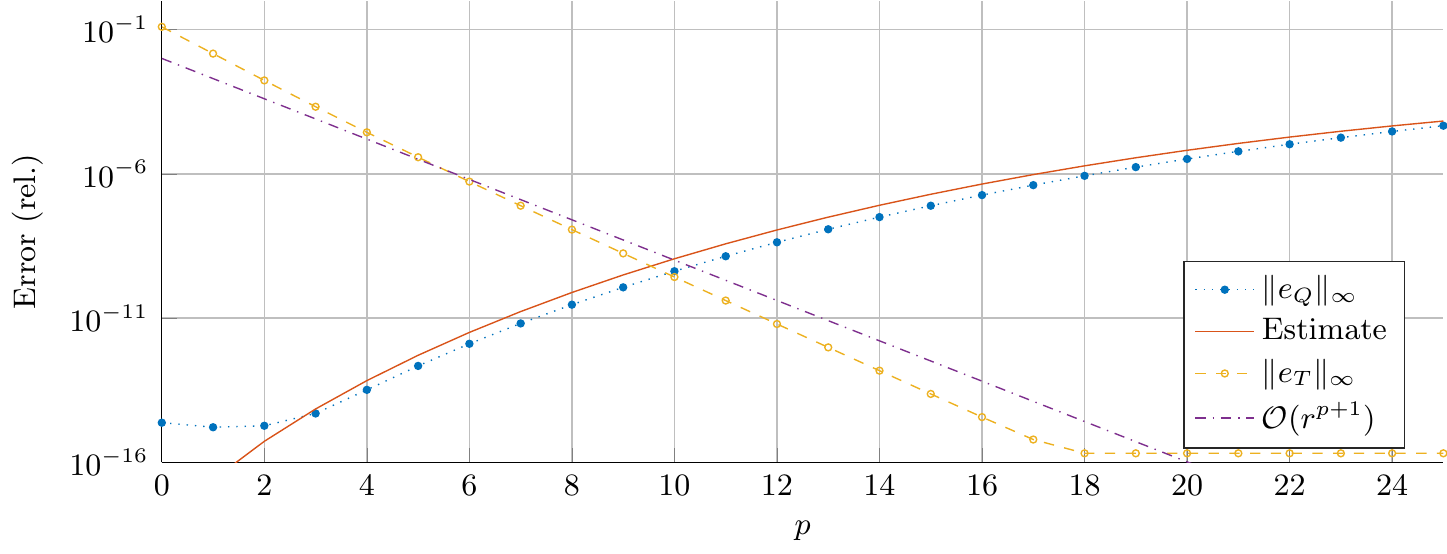}
  \caption{QBX error components \eqref{eq:qbx_errors_3d_trunc} and
    \eqref{eq:qbx_errors_3d_quad} measured on the surface of a
    spheroid, $a=1,c=2$, with $\sigma=1$, $r=0.2$ and a $400\times
    200$ grid ($n_t \times n_s$). Estimate computed using
    \eqref{eq:qbx_spheroid_final}.}
  \label{fig:3dqbx_spheroid_full}
\end{figure}

\end{appendices}

\clearpage
\ifsvjour
\bibliographystyle{spmpsci} 
\else
\bibliographystyle{jabbrv_abbrv_doi}
\fi
\bibliography{library}

\end{document}